\documentclass[]{article}
\usepackage[letterpaper,hmargin=27mm,vmargin=27mm,bmargin=27mm,tmargin=22mm]{geometry}
\usepackage[utf8]{inputenc}
\usepackage{amsmath,amsthm,amssymb}
\usepackage{mathtools}
\usepackage[T1]{fontenc}
\usepackage[english]{babel}
\usepackage{mathrsfs}
\usepackage{yfonts}
\usepackage{pgf,tikz}
\usetikzlibrary{arrows}
\usepackage{listings}
\usepackage{graphicx}
\usepackage{float}
\usepackage{caption}
\usepackage{subcaption}
\usepackage[normalem]{ulem}
\usepackage{changepage}
\usepackage{listings}
\usepackage[final]{pdfpages}
\usepackage{hyperref}
\usepackage{lipsum}
\usepackage{url}
\usepackage{algorithm}
\usepackage{algpseudocode}
\usepackage{pgfplots}
\pgfplotsset{compat=1.15}
\usepackage{mathrsfs}
\usepackage[toc]{glossaries}
\usepackage{afterpage}
\usetikzlibrary{arrows}
\usepackage{authblk}

\newtheorem{theorem}{Theorem}[section]
\newtheorem{prop}[theorem]{Proposition}
\newtheorem{corollary}[theorem]{Corollary}
\newtheorem{lemma}[theorem]{Lemma}

\newtheorem{remark}[theorem]{Remark}
\theoremstyle{definition}
\newtheorem{example}[theorem]{Example}
\newcommand{\norm}[1]{\left\lVert#1\right\rVert}
\newcommand{\abs}[1]{\left\lvert#1\right\rvert}
\DeclareMathOperator{\wind}{wind}
\DeclareMathOperator{\spec}{sp}

\DeclareMathOperator{\diag}{diag}

\newcommand{\di}{\mathop{}\!{d}}

\DeclarePairedDelimiter\floor{\lfloor}{\rfloor}
\newcommand{\code}[1]{\texttt{#1}}

\title{A geometric approach to approximating the limit set of eigenvalues for banded Toeplitz matrices}
\author{Teodor Bucht\footnote{Centre for Mathematical Sciences, Lund University, Box 118, 22100 Lund, Sweden, and KTH Royal Institute of Technology, Stockholm, Sweden. T. Bucht's research was supported in part by the Swedish Reseach Council(VR) Grant No.\;2021-04703.}\;
and Jacob S. Christiansen\footnote{Centre for Mathematical Sciences, Lund University, Box 118, 22100 Lund, Sweden. J. S. Christiansen's research was supported in part by the Swedish Research Council (VR) Grant No.\;2018-03500.}}

\makeglossaries
\begin{document}
	
	\maketitle
	\pagenumbering{arabic}

%

\begin{abstract}
	This article is about finding the limit set for banded Toeplitz matrices. Our main result is a new approach to approximate the limit set $\Lambda(b)$ where $b$ is the symbol of the banded Toeplitz matrix. The new approach is geometrical and based on the formula $\Lambda(b) = \cap_{\rho \in (0, \infty)} \spec T(b_\rho)$, where $\rho$ is a scaling factor, i.e. $b_\rho(t) := b(\rho t)$, and $\spec(\cdot)$ denotes the spectrum. We show that the full intersection can be approximated by the intersection for a finite number of $\rho$'s, and that the intersection of polygon approximations for $\spec T(b_\rho)$ yields an approximating polygon for $\Lambda(b)$ that converges to $\Lambda(b)$ in the Hausdorff metric. Further, we show that one can slightly expand the polygon approximations for $\spec T(b_\rho)$ to ensure that they contain $\spec T(b_\rho)$. Then, taking the intersection yields an approximating superset of $\Lambda(b)$ which converges to $\Lambda(b)$ in the Hausdorff metric, and is guaranteed to contain $\Lambda(b)$. Combining the established algebraic (root-finding) method with our approximating superset, we are able to give an explicit bound on the Hausdorff distance to the true limit set. We implement the algorithm in Python and test it. It performs on par to and better in some cases than existing algorithms. We argue, but do not prove, that the average time complexity of the algorithm is $O(n^2 + mn\log m)$, where $n$ is the number of $\rho$'s and $m$ is the number of vertices for the polygons approximating $\spec T(b_\rho)$. Further, we argue that the distance from $\Lambda(b)$ to both the approximating polygon and the approximating superset decreases as $O(1/\sqrt{k})$ for most of $\Lambda(b)$, where $k$ is the number of elementary operations required by the algorithm. 
\end{abstract}

\section{Introduction}

In this work we study the eigenvalues of $n \times n$ truncated infinite Toeplitz matrices \begin{equation}
	T_n(a) := \begin{pmatrix}
		a_0 & a_{-1} &  a_{-2} & \ldots  & a_{-n+1}\\
		a_1 & a_{0} &  a_{-1} & \ldots & a_{-n+2} \\
		a_2 & a_{1} &  a_{0} & \ldots & \ldots \\
		\ldots  & \ldots  &  \ldots  & \ldots & a_{-1} \\
		a_{n-1} & a_{n-2} & \ldots & a_1 & a_0
	\end{pmatrix},
	\label{generic_finite_toeplitz_matrix}
\end{equation} as $n \to \infty$. Here $a : \mathbb{T} \to \mathbb{C}$  is the associated \textit{symbol} of the matrix defined by \begin{equation*}
a(t) := \sum_{n = - \infty}^{\infty} a_n t^n, \quad t \in \mathbb{T},
\end{equation*} where $\mathbb{T}$ denotes the unit circle in $\mathbb{C}$. We regard $T_n(a)$ as a truncation of the infinite Toeplitz matrix \begin{equation*}
T(a) := (a_{j-k})_{j, k = 0}^{\infty}, 
\label{general_toeplitz}
\end{equation*} which naturally can be regarded as an operator on $l^2(\mathbb{N})$. In \cite{toeplitz_original} Toeplitz showed that $(a_{j-k})_{j, k = 0}^{\infty}$ is bounded on $l^2(\mathbb{N})$ if and only if there exists a function $a \in L^{\infty}(\mathbb{T})$ that has $(a_k)_{k=-\infty}^{\infty}$ as Fourier coefficients. 

A key idea for investigating $T_n(a)$ is that $T_n(a)$ should ``mimic'' the properties of $T(a)$, at least for large $n$. In particular, it should be possible to extract information about $\spec T_n(a)$ from $\spec T(a)$. However, the connection between $\spec T_n(a)$ and $\spec T(a)$ is quite intricate for non-self-adjoint Toeplitz matrices and in general $\spec T_n(a)$ does not approach $\spec T(a)$ as $n \to \infty$. Still, there is a result that links finite dimensional spectra with infinite dimensional spectra which we can use. It is stated below in \eqref{limit_set_two_classification}.

Recall the following result due to Gohberg \cite{gohberg1952application, gohberg1958} describing $\spec T(a)$ for continuous symbols: \begin{equation*}
\spec T(a) = a(\mathbb{T}) \cup \{\lambda \in \mathbb{C} \setminus a(\mathbb{T}) : \wind{(a-\lambda)} \neq 0\},
\end{equation*} where $\wind (a - \lambda)$ denotes the winding number of $a$ around $\lambda$. 

As for $T_n(a)$, denote its eigenvalues by $\lambda_j^{(n)}, \; j=1, 2, \ldots, n$, and define the Borel measures \begin{equation*}
	\mu_n(E) := \frac{\# \{j : \lambda_j^{(n)} \in E\}}{n},
\end{equation*} where $E \subset \mathbb{C}$ is any Borel subset. These measures describe which fraction of the eigenvalues of $T_n(a)$ are located in $E$. When studying eigenvalues of $T_n(a)$, a central question is whether there exists a limiting measure $\mu$ that $\mu_n$ converge to, at least weakly, as $n \to \infty$, and if it exists, we want to compute it. This limiting measure is also referred to as the limiting eigenvalue distribution. 

Another central problem when studying the eigenvalues of $T_n(a)$ is finding the strong and weak limit sets defined by \begin{equation*}
	\begin{split}
		\Lambda_{s}  & := \bigl\{\lambda \in \mathbb{C} : \lambda_{j_n}^{(n)} \to \lambda \text{ for some sequence } j_n \to \infty\bigr\} = \limsup \spec T_n(a),  \\
		\Lambda_w  & := \bigl\{\lambda \in \mathbb{C} : \lambda_{j_k}^{(n_k)} \to \lambda \text{ for some sequences } j_k, n_k \to \infty \bigr\} = \liminf \spec T_n(a).
	\end{split}
\end{equation*} Loosely speaking, $\Lambda_{s}$ and $\Lambda_w$ tell us where the eigenvalues cluster. We will occasionally use the notation $\Lambda_{s}(a)$ and $\Lambda_w(a)$ to signify that we mean the limit sets with respect to the symbol $a$. 

When the symbol is a real valued $L^\infty$ function, the first Szeg\H{o} limit theorem (see, e.g., \cite[Theorem 5.10]{bottcher1999introduction}) describes the limit sets and the limiting eigenvalue distribution. So the case of real-valued symbols is well understood. However, a Toeplitz matrix is generally not hermitian, and so the symbol is generally not real. Finding the limiting eigenvalue distribution and limit sets for arbitrary complex-valued symbols remains an open problem. 

The pioneers of dealing with such issues for complex-valued symbols were Schmidt and Spitzer. In \cite{schmidt_spitzer}, they investigated the subclass of banded Toeplitz matrices. A Toeplitz matrix is said to be banded if its symbol $b$ is a Laurent polynomial, that is, \begin{equation*}
	b(t) = \sum_{n = -r}^s b_n t^n,
\end{equation*} for some positive integers $r$ and $s$. This means that there only are a finite number of non-zero elements in each row of $T_n(b)$. Note that if  $r$ or $s$ were non-positive integers, $T_n(b)$ would be upper or lower triangular, and all the eigenvalues would hence be trivially equal to the diagonal element, and hence $\Lambda_{s}$ and $\Lambda_w$ would both be equal to a set consisting of one point. 

A crucial insight of Schmidt and Spitzer is that if $\Lambda_{s}(b)$ and $\Lambda_w(b)$ are to ``mimic'' the behavior of $\spec T(b)$, then they are forced to also mimic the behavior of $\spec T(b_\rho)$, where $b_\rho(t) := b(\rho t)$. For $T_n(b_\rho)$ and $T_n(b)$ only differ by a similarity transform, 
\begin{equation*}
	T_n(b_\rho) = \diag(\rho, \rho^2, \ldots, \rho^n) \; T_n(b) \; \diag(\rho^{-1}, \rho^{-2}, \ldots , \rho^{-n}).
\end{equation*} So $\Lambda_{s}(b_\rho) = \Lambda_{s}(b)$ and $\Lambda_w(b_\rho) = \Lambda_w(b)$. Consequently, $\Lambda_{s}(b)$ and $\Lambda_w(b)$ should also mimic \begin{equation}
\bigcap_{\rho \in (0, \infty)} \spec T(b_\rho).
\label{infinite_intersect}
\end{equation} 

Schmidt and Spitzer also defined the polynomial \begin{equation*}
	\begin{multlined}
	Q(\lambda, z) := z^r \bigl(b(z) - \lambda\bigr) = \\ \beta_{-r} + \beta_{-r + 1} z + \ldots + (\beta_0 - \lambda) z^r + \ldots + \beta_{s-1} z^{r+s-1} + \beta_s z^{r+s}
	\end{multlined}
\end{equation*} and showed that there is a connection between the roots of $Q(\lambda, z)$ and the set in \eqref{infinite_intersect}. In particular, if we for fixed $\lambda$ denote the roots of $Q(\lambda, z)$ as $z_j(\lambda), \; j = 1, 2, \ldots, r+s$, and order them so that $\abs{z_j(\lambda)} \leq \abs{z_{j+1}(\lambda)}, \; j=1, 2, \ldots, r+s-1$, Schmidt and Spitzer showed that \begin{equation}
\Lambda(b) := \bigl\{\lambda \in \mathbb{C} : \abs{z_r(\lambda)} = \abs{z_{r+1}(\lambda)}\bigr\} =  \bigcap_{\rho \in (0, \infty)} \spec T(b_\rho).
\label{limit_set_two_classification}
\end{equation} The main result of Schmidt and Spitzer was proving that \begin{equation*}
\Lambda_{s}(b) = \Lambda_w(b) = \Lambda(b),
\end{equation*} this also explains why $\Lambda(b)$ is simply called the \textit{limit set}. Additionally, they showed that $\Lambda(b)$ is a finite union of analytic arcs. A few years later, Ullman \cite{ullman} showed that $\Lambda(b)$ is a connected set. Further significant results on the asymptotic eigenvalue behavior for banded Toeplitz matrices include Hirschman \cite{hirschman}, Widom \cite{widom1990eigenvalue, widom1994} and Duits and Kuijlaars \cite{equilibrium_problem_banded}. Hirschman derived the limiting eigenvalue distribution for banded matrices, Widom simplified the proofs of both Schmidt and Spitzer and Hirschman, and Duits and Kuijlaars characterized the limiting eigenvalue distribution in terms of an equilibrium problem. 

Also, noteworthy monographs on the subject of Toeplitz matrices include Böttcher and Silbermann \cite{bottcher1999introduction} and Böttcher and Grudsky \cite{spectral_prop_bottcher}. In \cite{bottcher1999introduction}, the general theory of Toeplitz matrices is presented, whilst \cite{spectral_prop_bottcher} focuses on spectral properties of banded Toeplitz matrices. 


\label{introduction}
As explained, there is for banded Toeplitz matrices a well developed theory that describes the structure of the limit set $\Lambda(b)$ and how dense the eigenvalues cluster on $\Lambda(b)$. But how do we find $\Lambda(b)$ for practical examples? A natural approach to finding $\Lambda(b)$ is to calculate $\spec T_n(b)$ for $n$ of increasing size and simply observe where the eigenvalues seem to cluster. However, in general this approach does not work. Consider for instance $b(t) = t^{-4} + t^1$. Then one can show that \begin{equation*}	
	\Lambda(b) = \bigl\{\lambda \in \mathbb{C} : \lambda = re^{\frac{2\pi i}{5}}, \, 0 \leq r \leq 5 \cdot 4^{-4/5}\bigr\}.
\end{equation*} Additionally, one can show that $\spec T_n(b) \subset \Lambda(b)$ for all $n$. We compare $\Lambda(b)$ with the computed eigenvalues of $T_{200}(b)$ and $T_{400}(b)$ using NumPys \code{numpy.linalg.eig}. The results are shown in Figure \ref{example_bad_numerical_behaviour}, where we see that the eigenvalues do not cluster around $\Lambda(b)$ and even seem to stray away from $\Lambda(b)$ which is opposite to what we would expect. So what we see in Figure \ref{example_bad_numerical_behaviour} is likely a result of rounding errors, meaning that the algorithm instead computes the eigenvalues for a nearby matrix.

\begin{figure}[h!]
	\centering
	\begin{subfigure}[b]{0.49\textwidth}
		\centering
		\includegraphics[width=\textwidth]{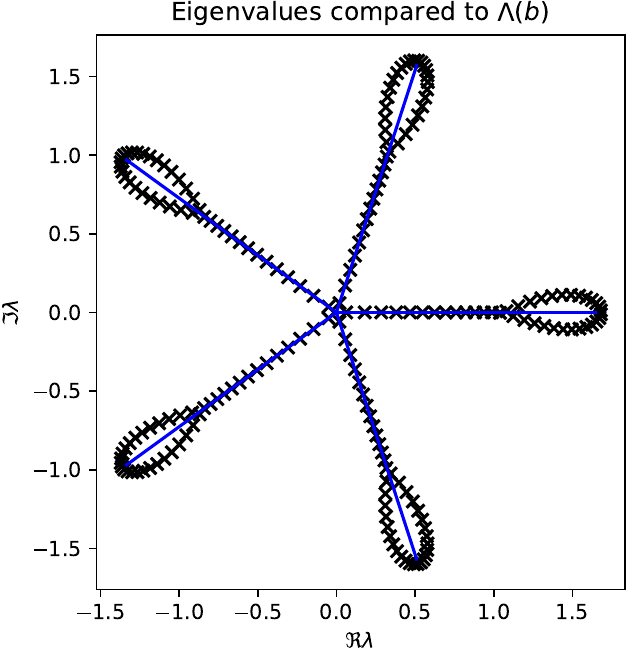}
		\caption{$\Lambda(b)$ and the computed eigenvalues of $T_{200}(b)$ for $b(t)=t^{-4}+t$.}
		\label{num_bad_conv_200}
	\end{subfigure}
	\hfill
	\begin{subfigure}[b]{0.49\textwidth}
		\centering
		\includegraphics[width=\textwidth]{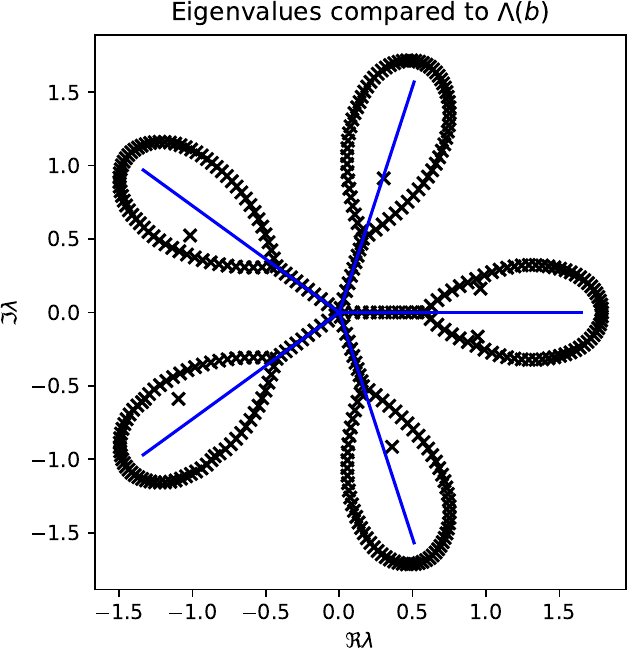}
		\caption{$\Lambda(b)$ and the computed eigenvalues of $T_{400}(b)$ for $b(t)=t^{-4}+t$.}
		\label{num_bad_conv_400}
	\end{subfigure}
	\caption{Examples of rounding error behavior for computed eigenvalues.}
	\label{example_bad_numerical_behaviour}
\end{figure}

In the recent article \cite{recent_structure}, it is stated that ``Finding this limiting set [refering to $\Lambda(b)$] nevertheless remains a challenge''. Then the authors propose an algorithm that is based on \cite{beam_warming}, which essentially looks at the roots of $Q(z, \lambda)$. This algorithm produces a set of points contained in the limit set and this subset of points, which we from here on denote the \textit{approximating subset} approaches $\Lambda(b)$ in the Hausdorff metric (see \eqref{hausdorff_metric_def} below) as the number of sampling points tend to infinity. The ``backbone'' of this algorithm is polynomial rootfinding. Instead of searching for $\Lambda(b)$ algebraically, we present a novel geometric approach, where the main computational task is computing polygon intersections. 

The idea behind our new approach presented in this article is to approximate $\Lambda(b)$ using \eqref{infinite_intersect}. To do this, we first show that only a compact set of $\rho$'s is needed in the intersection, and then we introduce polygons approximating $\spec T(b_\rho)$, which we denote by $\spec T(b_\rho^D)$. We show that the intersection of $\spec T(b_{\rho_i}^D)$ for a finite number of sampled $\rho_i$'s converges to $\Lambda(b)$ in the Hausdorff metric as the number of sampled $\rho_i$'s tend to infinity. This gives us a way to approximate $\Lambda(b)$; we first find the compact interval of $\rho$'s needed in \eqref{infinite_intersect}, then sample $\rho$'s in this interval, construct polygon approximations for $\spec T(b_{\rho_i})$, and finally compute the intersection which yields an \textit{approximating polygon} for $\Lambda(b)$. 

Further, we show that one can slightly expand the polygon approximations $\spec T(b_{\rho_i}^D)$ to make sure that the expanded polygons contain $\spec T(b_\rho)$. Computing the intersection of the expanded polygons then yields an \textit{approximating superset} of $\Lambda(b)$ that (contains $\Lambda(b)$ and) approaches $\Lambda(b)$ in the Hausdorff metric. Combining the approximating subset of the algebraic approach and the approximating superset from our geometric algorithm, we are able to give an explicit bound on the Hausdorff distance to $\Lambda(b)$.  Additionally, we argue but do not prove that the Hausdorff distance between $\Lambda(b)$ and the approximating superset should decrease as $O(1/\sqrt{k})$, where $k$ is the number of basic operations needed to run the algorithm. 

A related and very important topic when considering spectra of Toeplitz matrices is \textit{pseudospectra}. The $\epsilon$-\textit{pseudospectrum} $\spec_\epsilon A$ of a bounded linear operator $A$ on a Banach space can be defined as \begin{equation*}
	\spec_\epsilon A := \left\{\lambda \in \mathbb{C} : \norm{(A - \lambda I)^{-1}} \geq 1/\epsilon \right\},
\end{equation*} where the convention that $\norm{(A - \lambda I)^{-1}} = \infty$ if $\lambda \in \spec A$ is used. Whilst the spectrum of $T_n(a)$ can behave quite differently from the spectrum of $T(a)$, which especially is the case for banded symbols, the pseudospectrum is much more well behaved. Most notably we have for example that if $a$ is continuous, then \begin{equation}
	\lim_{n \to \infty} \spec_\epsilon T_n(a) = \spec_\epsilon T(a),
		\label{pseudospectra_convergence}
 \end{equation} with respect to the Hausdorff metric, see Böttcher \cite{pseudospectra_and_singular_values}. Also, for matrices with banded symbol $b$ it is known that the resolvent grows exponentially in $n$ for points around which $b(\mathbb{T})$ has nonzero winding number. 
 
 The eigenvalues of $T_n(a)$ give us information about e.g. the asymptotic stability of $e^{tT_n(a)}$, but in practice the pesudospectrum of $T_n(a)$ gives us much better information about $\norm{e^{tT_n(a)}}$, and especially when analyzing stability of non-linear systems it is of great importance to consider the pseudospectrum. Also, the behavior in Figure \ref{example_bad_numerical_behaviour} can be further understood by considering the pseudospectrum. One can show that if we consider a bounded linear operator $A$ on a Hilbert space $H$, i.e. $A \in \mathcal{B}(H)$, then \begin{equation*}
	\spec_\epsilon(A) = \bigcup_{E \in \mathcal{B}(H) :  \norm{E} \leq \epsilon} \spec{(A+E)}.
 \end{equation*} So, whilst the computed eigenvalues will not be in the spectrum of $T_n(b)$, they will be in $\spec_\epsilon(T_n(b))$ for $\epsilon$ chosen as the floating point rounding error. 
 
 For an informative introduction to pseudospectra, see \cite[Chapter~23]{hogben} written by Embree. Also see \cite[Chapter~3]{bottcher1999introduction} for a more thorough treatment of pseudospectra related to Toeplitz matrices and \cite{embree_trefethen} by Trefethen and Embree for a broad treatment of pseudospectra. 
  
Recently there has been interesting developments aiming at approximating the asymptotic distribution of the eigenvalues of $T_n(a)$ on $\Lambda(a)$ as $n \to \infty$. In \cite{matrix_less_real_eigs}, Ekström and Vassalos propose an algorithm for approximating the asymptotic distribution of eigenvalues in the case where $T_n(a)$ has real eigenvalues for all $n$ (note that this is not equivalent to $a$ being real valued) And in \cite{matrix_less_complex_eigenvals}, Bogoya et al. extend this idea to complex eigenvalues, notably including the case of banded Toeplitz matrices where $\Lambda(b)$ consists of only one analytic arc.

\subsection{Previous work on computing $\Lambda(b)$}
\label{previous_work_subsection}
A naive approach to calculating $\Lambda(b)$ is to sample a grid in $\mathbb{C}$ and check what points almost are in the set, by solving $b(z) = \lambda$ and looking at the root sizes. However, this approach is time consuming and it is not obvious how a point ``almost'' being in $\Lambda(b)$ should be interpreted. 

In \cite{recent_structure}, a better algorithm for calculating $\Lambda(b)$ is presented. The idea is: instead of sampling $\lambda$, sample $\varphi \in (0, 2 \pi)$ and solve \begin{equation}
	b(z) - b(ze^{i\varphi}) = 0. 
	\label{graph_intersect_property}
\end{equation} This approach works since if $\lambda \in \Lambda(b)$ then two of the roots of $b(z) - \lambda = 0$, $z_r(\lambda)$ and $z_{r+1}(\lambda)$ have the same modulus, so $z_{r+1}(\lambda)/z_{r}(\lambda) = e^{i \varphi}$, which implies that $z_r(\lambda)$ is a root of \eqref{graph_intersect_property}. Note that in the case $z_r(\lambda) = z_{r+1}(\lambda)$, \eqref{graph_intersect_property} is the equation $0=0$. But for double roots, the derivative also has a zero, so to find the cases where $z_r(\lambda)  = z_{r+1}(\lambda)$, one solves \begin{equation}
b'(z) = 0. 
\label{deriv_zero}
\end{equation} From solving \eqref{deriv_zero}, and \eqref{graph_intersect_property} for the sampled $\varphi$'s, we get some candidates $z_k$. For each of these candidates, we can calculate the corresponding $\lambda_k := b(z_k)$ and check each of these $\lambda_k$ for membership in $\Lambda(b)$ by solving the equation \begin{equation*}
b(z) = \lambda_k 
\end{equation*} and checking if the $r$'th and $(r+1)$'th smallest roots (when ordered by modulus) have the same absolute value. 

The output from this algorithm is a set of points belonging to $\Lambda(b)$, and it involves numerically finding the roots of multiple polynomials that have degree $r+s$. For fixed $r$ and $s$, the computational time is at most $O(N)$, where $N$ is the amount of sampled $\varphi$'s. Note that we are assuming that polynomial root finding is a fixed time process for a fixed degree polynomial. 

\begin{example}
	Consider the symbol \begin{equation}
		b(t) = -2t^{-1} + 4(1-i) + 7it -3(1+i)t^2 + t^3.
		\label{main_example_eq}
	\end{equation} This symbol is presented in \cite{recent_structure} as an example of where the numerically computed eigenvalues for $T_n(b_\rho)$ do not converge to the true $\Lambda(b)$, and where the computed eigenvalues vary with $\rho$, which they theoretically do not. We estimated $\Lambda(b)$ using the algebraic approach just presented. The results can be seen in Figure \ref{algebraic_alg_all_results}. 
	It took roughly 3 minutes to run the algorithm. For solving polynomials, our implementation uses NumPys standard implementation \code{numpy.roots} that computes the eigenvalues of the companion matrix. Also, when checking if the absolute values of the roots are equal, we accept an error of $10^{-7}$. 

\begin{figure}[H]
	\centering
	\begin{subfigure}[t]{0.49\textwidth}
		\centering
		\includegraphics[width=\textwidth]{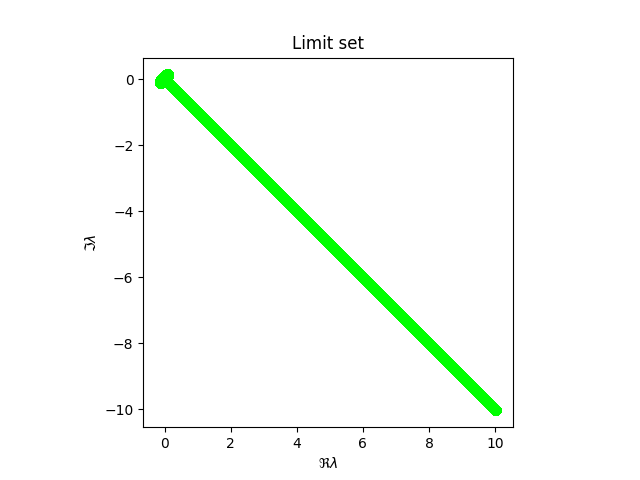}
		\caption{$\Lambda(b)$  estimated using the algebraic approach (from section \ref{previous_work_subsection}) with $2\cdot 10^5$ $\varphi$'s sampled uniformly in $[0, \pi]$.}
		\label{algebraic_alg_results}
	\end{subfigure}
	\hfill
	\begin{subfigure}[t]{0.49\textwidth}
		\centering
		\includegraphics[width=\textwidth]{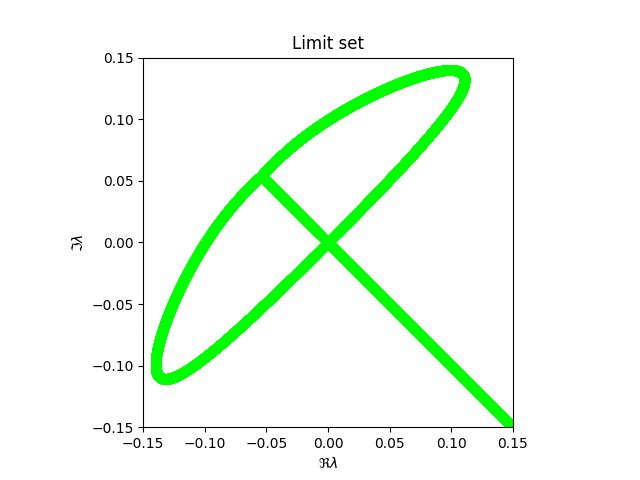}
		\caption{A zoomed in version of \ref{algebraic_alg_results}.}
		\label{algebraic_alg_results_zoom}
	\end{subfigure}
	
	\caption{The results from estimating $\Lambda(b)$ using the algebraic algorithm presented above with $b$ as in \eqref{main_example_eq}.}
	\label{algebraic_alg_all_results}
\end{figure}
\label{main_example}
\end{example}

All the code used for computing limit sets, both the algebraic approach just presented and our new geometric approach, is available at \cite{glisec}. 


\subsection{Outline}
This article is organized as follows. In Section \ref{theory_section} we present the theoretic background of our algorithm, and show that the approximating polygon and approximating superset being produced both converge to $\Lambda(b)$. In Section \ref{implementation_details_section} we discuss how the algorithm can be implemented. We also suggest a better way of sampling $\rho$'s for the intersection, and compare the naive and improved way of sampling $\rho$'s. In Section \ref{time_complexity_conv_speed_section} we reason about the time complexity and convergence speed of the algorithm. Lastly, we provide some concluding remarks in Section \ref{conclusion_section}.  

\section{A geometric approach}

\label{theory_section}

One way to look at the problem of finding $\Lambda(b)$ is the result of Schmidt and Spitzer \cite{schmidt_spitzer} presented in \eqref{limit_set_two_classification}, which says \begin{equation}
	\bigcap_{\rho \in (0, \infty)} \spec T(b_\rho) = \Lambda(b).
	\label{left_vs_right}
\end{equation} We can either work with the left hand side and try to compute it geometrically, or look at the right hand side and compute it algebraically, which is what previous works have done, see subsection \ref{previous_work_subsection}. We will now show a geometric approach to approximating $\Lambda(b)$, which involves approximating the left-hand side in \eqref{left_vs_right}.

The first observation that speaks in favor of $\cap_{\rho \in (0, \infty)} \spec T(b_\rho)$ even being possible to compute numerically is that not all $\rho$'s are needed. 
\begin{prop}
	Given a Laurent polynomial $b$, there exists $\rho_l, \rho_h$ with $0 < \rho_l \leq \rho_h < \infty$ such that \begin{equation}
		\Lambda(b) = \bigcap_{\rho \in (0, \infty)} \spec T(b_\rho) = \bigcap_{\rho \in [\rho_l, \rho_h]} \spec T(b_\rho).
		\label{compact_rhos_eq}
	\end{equation} Furthermore, examples of such $\rho_l$ and $\rho_h$ can be found by solving a polynomial equation with real coefficients. 
\label{compact_rhos_theorem}
\end{prop}
\begin{proof}
	The idea behind the proof is that for small enough $\rho$ the lowest order term in $b$ will dominate, so then $b_\rho(\mathbb{T})$ will roughly trace out $r$ circles of radius $\beta_{-r}\rho^{-r}$ winding in negative direction. Thus, intersecting with $\spec T(b_\rho)$ for small enough $\rho$ does not erase any new points. The same argument holds for large enough $\rho$. 
	
	Now, set $K := \norm{b_1}_\infty$ and choose $0 < \rho_l \leq 1 \leq \rho_h$ such that $\spec T(b_\rho) \supset B(0, K)$ for $\rho < \rho_l$ and $\rho_h < \rho$, where $B(0, K)$ is the closed ball centered at $0$ with radius $K$. Then we have \begin{equation*}
		\bigcap_{\rho \in (0, \infty)} \spec T(b_\rho) \supset B(0, K) \cap \Biggl(\,\bigcap_{\rho \in [\rho_l, \rho_h]} \spec T(b_\rho)\Biggr) = \bigcap_{\rho \in [\rho_l, \rho_h]} \spec T(b_\rho) \supset \Lambda(b),
	\end{equation*}
	where the equality holds since $\spec T(b_1) \subset B(0, K)$ and $\rho_l \leq \rho_h$. As the left-most and right-most sets are equal in the chain of inclusions, all sets must be equal, showing our desired result. 
	
	Now we motivate how $\rho_l$ and $\rho_h$ with the correct properties can be found. As before, $K = \norm{b_1}_\infty$. Let $\rho_l$ be the minimum of 1 and the smallest real positive solution to $\abs{\beta_{-r}}\rho^{-r} - \sum_{n=-r+1}^{s} \abs{\beta_n} \rho^n = K$. Hence $\abs{\beta_{-r}}\rho^{-r} - \sum_{n=-r+1}^{s} \abs{\beta_n} \rho^n > K$ for all $\rho < \rho_l$. We now estimate the size of $b_\rho(t)$. For $\rho$ in $0 < \rho < \rho_l$ and $t \in \mathbb{T}$, we have \begin{equation*}
		\begin{split}
			\abs{b_\rho(t)} = \abs{\sum_{n = -r}^{s} \beta_n \rho^n t^n} & \geq \abs{\beta_{-r}} \rho^{-r} - \abs{\sum_{n=-r+1}^{s}\beta_n \rho^n t^n} \\ & \geq \abs{\beta_{-r}}\rho^{-r} - \sum_{n=-r+1}^{s} \abs{\beta_n} \rho^n > K.
		\end{split}
	\end{equation*} Hence $\spec T(b_\rho) \supset B(0, K)$ for $\rho$ in $0 < \rho < \rho_l$. Analogously we can define $\rho_h$ as the maximum of 1 and the largest positive root to $\abs{\beta_{s}}\rho^{s} - \sum_{n=-r}^{s-1} \abs{\beta_n} \rho^n = K$ and do similar calculations to get that $\spec T(b_\rho) \supset B(0, K)$ for $\rho > \rho_h$, which concludes our proof. 
\end{proof}

So, only a compact set of $\rho$'s is needed to obtain $\Lambda(b)$. A natural approach to approximating $\Lambda(b)$ is to select a \textit{partition} $\rho_l = \rho_0 < \rho_1 < \ldots < \rho_{n-1} < \rho_{n} = \rho_h$ and compute \begin{equation}
	\bigcap_{j=0}^{n} \spec T(b_{\rho_j}).
	\label{finite_intersects}
\end{equation} The \textit{granularity} of the partition $(\rho_j)_{j=0}^{n}$ is defined as $$\Delta := \max_{0 \leq j \leq n-1} (\rho_{j+1} - \rho_j).$$ We will denote the set of all partitions of $[a, b]$ with granularity $\leq \Delta$ as $\mathcal{P}_{[a, b]}(\Delta)$. The question is now if \eqref{finite_intersects} will approach $\Lambda(b)$, and in what way? In what follows, we will show that \eqref{finite_intersects} approaches $\Lambda(b)$ in the \textit{Hausdorff metric} as the granularity of $(\rho_j)_{j=0}^n$ tends to $0$. Recall that the Hausdorff metric, denoted $d_H$, operates on the set $\mathcal{C}$ consisting of all the bounded non-empty closed subsets of $\mathbb{C}$. It is defined by \begin{equation}
d_H(A, B) := \inf \left\{\epsilon \, : \, A \subset (B)_\epsilon, B \subset (A)_\epsilon\right\},
\label{hausdorff_metric_def}
\end{equation} where $(X)_\epsilon$ denotes the \textit{$\epsilon$--fattening} of $X$, that is, \begin{equation*}
(X)_\epsilon := \bigcup_{x \in X} \left\{z \in \mathbb{C} \, : \, \abs{z -x} \leq \epsilon\right\}.
\end{equation*} As is well known, $(\mathcal{C}, d_H)$ forms a complete metric space, see e.g. \cite[p.\,121--124]{hausdorff_reference}. 

To show that \eqref{finite_intersects} approaches $\Lambda(b)$, we need the following lemma.
\begin{lemma}
	For a Laurent polynomial $b$, the map \begin{equation*}
		\rho \mapsto \spec T(b_\rho)
	\end{equation*} from $ (\mathbb{R}^+, |\cdot|) \to (\mathcal{C}, d_H)$ is continuous. 
\label{spec_continuous}
\end{lemma}
\begin{proof}
	Let $(\mathcal{T}_C, \norm{\,\cdot\,})$ denote the space of Toeplitz operators with continuous symbol endowed with the operator norm. We see that our map is the composition of three maps, \begin{equation}
		\begin{split}
			\rho \mapsto b_\rho, & \quad (\mathbb{R}^+, |\cdot|) \to (C(\mathbb{T}), ||\cdot||_\infty), \\
			a \mapsto T(a), & \quad (C(\mathbb{T}), ||\cdot||_\infty) \to (\mathcal{T}_C, ||\cdot||), \\
			T(a) \mapsto \spec T(a), & \quad (\mathcal{T}_C, ||\cdot||) \to (\mathcal{C}, d_H).
		\end{split}
	\label{cont_composition}
	\end{equation}
The first and second maps in \eqref{cont_composition} are clearly continuous, and by \cite[Theorem 10]{cont_spectrum} the third map is also continuous. 
\end{proof}

\begin{remark}
	Recall that the map $A \mapsto \spec A$, from the set of bounded operators on a Hilbert space to $(\mathcal{C}, d_H)$ is not continuous is general, see e.g. \cite[Example 3.10]{perturbation_theory_kato}. According to \cite{cont_spectrum}, it was conjectured however that $A \mapsto \spec A$ is continuous on the manifold of Toeplitz operators. But this conjecture was proven false in \cite{spectrum_not_cont}, where the authors constructed symbols $a_n, a \in L^\infty(\mathbb{T})$ such that $\norm{a_n - a}_\infty \to 0$ but $d_H(\spec T(a_n), \spec T(a)) \not\rightarrow 0$. It should be noted that continuity is proven for a much larger class of symbols than $C(\mathbb{T})$ in \cite{cont_spectrum}.
\end{remark}

Now we can use Lemma \ref{spec_continuous} to show the promised result, that \eqref{finite_intersects} approaches $\Lambda(b)$ in the limit. 

\begin{theorem}
	Let $\rho_l$ and $\rho_h$ be bounds as in Theorem \ref{compact_rhos_theorem}. Then we have \begin{equation*}
		\lim_{\Delta \to 0} \sup_{\boldsymbol{\rho} \in \mathcal{P}(\Delta)} d_H\Biggl( \, \bigcap_{j=0}^{n} \spec T(b_{\rho_j}), \Lambda(b) \Biggr)  = 0.
	\end{equation*} 
where $\boldsymbol{\rho} = (\rho_j)_{j=0}^{n} $ and $\mathcal{P}(\Delta) = \mathcal{P}_{[\rho_l, \rho_h]}(\Delta)$.
 	\label{riemann_sum_like}
\end{theorem}
\begin{proof}
	From Lemma \ref{spec_continuous} we have that the map $\rho \mapsto \spec T(b_\rho)$ from $[\rho_l, \rho_h] \to \mathcal{C}$ is uniformly continuous, since $[\rho_l, \rho_h]$ is compact. So for a given $\epsilon > 0$ there is a $\delta > 0$ such that for all $\rho, \rho^* \in [\rho_l, \rho_h]$ with $\abs{\rho - \rho^*} < \delta$ we have $d_H(\spec T(b_{\rho^*}), \spec T(b_\rho)) < \epsilon$. It thus suffices to show that for all partitions $(\rho_j)_{j=0}^{n}$ with granularity $\Delta < 2\delta$ we have $$d_H\Biggl(\,\bigcap_{j=0}^{n} \spec T(b_{\rho_j}), \bigcap_{\rho \in [\rho_l, \rho_h]} \spec T(b_{\rho})\Biggr) \leq \epsilon.$$ Naturally we always have $\bigcap_{\rho \in [\rho_l, \rho_h]} \spec T(b_{\rho}) \subset \left(\bigcap_{j=0}^{n} \spec T(b_{\rho_j})\right)_\epsilon$. So we need only to prove the opposite inclusion. To do this, take an arbitrary \begin{equation*}
		 \lambda \notin \Biggl(\,\bigcap_{\rho \in [\rho_l, \rho_h]} \spec T(b_{\rho})\Biggr)_\epsilon = \bigcap_{\rho \in [\rho_l, \rho_h]} \bigl( \spec T(b_{\rho})\bigr)_\epsilon.
	\end{equation*} So there is a $\rho^* \in [\rho_l, \rho_h]$ such that $\lambda \notin (\spec T(b_{\rho^*}))_\epsilon$. Since $\Delta < 2 \delta$ there is a $\rho_k$ in the partition with $\abs{\rho^* - \rho_k} < \delta$. This implies $d_H(\spec T(b_{\rho^*}), \spec T(b_\rho)) < \epsilon$, which means $(\spec T(b_{\rho^*}))_\epsilon \supset \spec T(b_{\rho_k})$. So $\lambda \notin \spec T(b_{\rho_k})$, and we are done. 
\end{proof}

Theorem \ref{riemann_sum_like} hints at great possibilities for computing $\Lambda(b)$ numerically. But it is difficult to represent $\spec T(b_\rho)$ numerically. A possible solution is to approximate $\spec T(b_\rho)$ by a polygon with vertices sampled along $b_\rho(\mathbb{T})$. In what follows we shall show that such a polygon approximation indeed gives the desired result. We will partition $[0, 2\pi]$ with $0 = v_0 < v_1 < \ldots < v_{n-1} < v_{n} = 2 \pi$. The granularity for $(v_j)_{j=0}^{n}$ is defined as before. A partition of $[0, 2\pi]$ defines a polygon discretization $b_\rho^D$ approximating $b_\rho$, of the form $b_\rho^D : \mathbb{T} \to \mathbb{C}$, \begin{equation}
	b_\rho^D(e^{iv}) := \frac{v-v_j}{v_{j+1}-v_j} \left(b_\rho(e^{iv_{j+1}}) - b_\rho(e^{iv_{j}})\right) + b_\rho( e^{iv_{j}}), \; \; v \in [v_j, v_{j+1}).
	 \label{polygon_definition}
\end{equation} Below, when we write $b_\rho^D$ it will always be with respect to a partition $(v_j)_{j=0}^n$ which should be clear from the context. Since $b_\rho^D$ is a continuous symbol, we have \begin{equation*}
\spec T(b_\rho^D) = b_\rho^D(\mathbb{T}) \cup \left\{\lambda \in \mathbb{C} : \wind(b_\rho^D - \lambda) \neq 0\right\}.
\end{equation*} Naturally, we wish to bound the Hausdorff distance between the spectrum of $T(b_\rho)$ and $T(b_\rho^D)$. 

\begin{prop}
	Let $b$ be a Laurent polynomial and $b_\rho^D$ a polygon discretization of $b_\rho$ as is \eqref{polygon_definition}. Then \begin{equation*}
		d_H\bigl(\spec T(b_\rho), \spec T(b_\rho^D)\bigr) \leq \norm{b_\rho - b_\rho^D}_\infty.
	\end{equation*}
\label{expanded_polygon_prop}
\end{prop}

\begin{proof}
	Take an arbitrary $\delta_\rho > \norm{b_\rho - b_\rho^D}_\infty$. Our goal is to show $(\spec T(b_\rho))_{\delta_{\rho}} \supset \spec T(b_\rho^D)$ and $(\spec T(b_\rho^D))_{\delta_{\rho}} \supset \spec T(b_\rho)$, which we can do through a geometric argument. We will proceed by taking $\lambda$ either $\notin (\spec T(b_\rho))_{\delta_{\rho}}$ or $\notin (\spec T(b_\rho^D))_{\delta_{\rho}}$ and aim at proving that then $\lambda \notin \spec T(b_\rho^D)$ or $\lambda \notin \spec T(b_\rho)$, respectively. Since $a({\mathbb{T}}) \subset \spec T(a)$ for any continuous symbol $a$, we have by choice of $\delta_\rho$ that $\inf_{t \in \mathbb{T}} \abs{b_\rho(t) - \lambda} > 0$ and $\inf_{t \in \mathbb{T}} \abs{b_\rho^D(t) - \lambda} > 0$. Due to this, we can parameterize \begin{equation*}
		b_\rho(e^{iv})  - \lambda = r(v)e^{i\theta(v)} \; \text{ and } \; b_\rho^D(e^{iv}) - \lambda = r_D(v)e^{i\theta_D(v)},
	\end{equation*} with continuous $r, r_D, \theta, \theta_D$. Recall that $b_\rho^D(e^{iv_j}) = b_\rho(e^{iv_j})$, where $(v_j)_{j=0}^n$ is the partition of $[0, 2\pi]$ defining $b_\rho^D$. The winding numbers can thus be written as \begin{equation}
			\begin{split}
					\wind (b_\rho^D - \lambda) & = \frac{1}{2\pi}\bigl(\theta_D(v_n) - \theta_D(v_{0})\bigr) = \frac{1}{2\pi}\sum_{j=1}^n \theta_D(v_j) - \theta_D(v_{j-1}), \\
					\wind (b_\rho - \lambda) & = \frac{1}{2\pi}\bigl(\theta(v_n) - \theta(v_{0})\bigr) = \frac{1}{2\pi}\sum_{j=1}^n \theta(v_j) - \theta(v_{j-1}).
				\end{split}
			\label{winding_number_comparision}
	\end{equation} Our strategy is to show that $\theta_D(v_j) - \theta_D(v_{j-1}) = \theta(v_j) - \theta(v_{j-1})$ for all $j$. This will then imply that $\wind(b_\rho^D - \lambda) = \wind(b_\rho - \lambda)$ and we arrive at the desired conclusion. Begin with $j = 1$, and assume that $\theta(0) = \theta_D(0)$. Denote $b_\rho^D(v_0)$ by $A$ and  $b_\rho^D(v_1)$ by $B$ and the line through $A$ and $B$ by $l$. There are two cases to consider: the first case is when $A \neq B$  and the projection $P$ of $\lambda$ onto $l$ lies on the segment $AB$ (see Figure \ref{ill_first_case}), the second case is when $A = B$ or the projection $P$ of $\lambda$ onto $l$ does not lie on the segment $AB$ (see Figure \ref{ill_second_case}).

\begin{figure}
	\centering
	\begin{subfigure}[t]{0.49\textwidth}
		\definecolor{xdxdff}{rgb}{0.49,0.49,1}
		\definecolor{uququq}{rgb}{0.25,0.25,0.25}
		\definecolor{qqqqff}{rgb}{0,0,1}
		\begin{tikzpicture}[line cap=round,line join=round,>=triangle 45,x=0.6cm,y=0.6cm]
			\clip(4.26,-1.31) rectangle (12.97,4.69);
			\draw (4.62,1.3)-- (12.53,4.11);
			\draw [domain=8.19341856261866:17.706811897606272, color=red] plot(\x,{(--70.22-7.99*\x)/2.84});
			\draw [shift={(7.93,4.52)}] plot[domain=3.91:6.2,variable=\t]({1*4.62*cos(\t r)+0*4.62*sin(\t r)},{0*4.62*cos(\t r)+1*4.62*sin(\t r)});
			\draw [dashed, domain=5:8.19341856261866, color=black] plot(\x,{(--70.22-7.99*\x)/2.84});
			\begin{scriptsize}
				\fill [color=black] (4.62,1.3) circle (1.5pt);
				\draw[color=black] (4.7,1.63) node {$B$};
				\fill [color=black] (12.53,4.11) circle (1.5pt);
				\draw[color=black] (12.6,4.44) node {$A$};
				\fill [color=black] (8.19,1.68) circle (1.5pt);
				\draw[color=black] (8.47,1.81) node {$\lambda$};
				\fill [color=black] (7.91,2.47) circle (1.5pt);
				\draw[color=black] (8.1,2.79) node {$P$};
				\draw[color=black] (9.62,-1.1) node {$r_\lambda$};
				\fill [color=black] (8.8,-0.02) circle (1.5pt);
				\draw[color=black] (8.97,0.31) node {$C$};
				\draw[color=black] (11.50,0.31) node {$b_\rho([v_0, v_1])$};
				\draw[color=black] (9.9,3.8) node {$b^D_\rho([v_0, v_1])$};
			\end{scriptsize}
		\end{tikzpicture}
		\caption{Illustration of the first case.}
		\label{ill_first_case}
	\end{subfigure}
	\hfill
	\begin{subfigure}[t]{0.49\textwidth}
		\definecolor{ffffff}{rgb}{1,1,1}
		\definecolor{xdxdff}{rgb}{0.49,0.49,1}
		\definecolor{qqqqff}{rgb}{0,0,1}
		\begin{tikzpicture}[line cap=round,line join=round,>=triangle 45,x=0.45cm,y=0.45cm]
			\clip(3.9,-2.0) rectangle (15.95,5);	
			\draw [domain=13.801320330119172:15.95, color=red] plot(\x,{(--10.61-0.56*\x)/0.87});
			
			\draw [rotate around={-173.6:(9.49,2.4)}] (9.49,2.4) ellipse (5.24*0.45cm and 1.72*0.45cm);
			\fill[line width=1.2pt,color=ffffff,fill=ffffff,fill opacity=1.0] (4.62,1.3) -- (12.53,4.11) -- (12.59,4.93) -- (3.17,3.55) -- (3,-0.83) -- cycle;
			\draw (4.62,1.3)-- (12.53,4.11);
			\draw [dashed, domain=8:13.801320330119172, color=black] plot(\x,{(--10.61-0.56*\x)/0.87});
			\begin{scriptsize}
				\fill [color=black] (4.62,1.3) circle (1.5pt);
				\draw[color=black] (4.7,1.63) node {$B$};
				\fill [color=black] (12.53,4.11) circle (1.5pt);
				\draw[color=black] (12.6,4.44) node {$A$};
				\fill [color=black] (13.8,3.29) circle (1.5pt);
				\draw[color=black] (13.5,3.12) node {$\lambda$};
				\fill [color=black] (14.67,2.73) circle (1.5pt);
				\draw[color=black] (15.00,2.95) node {$C$};
				\draw[color=black] (15.50,1.9) node {$r_\lambda$};
				\draw[color=black] (10.50,0.3) node {$b_\rho([v_0, v_1])$};
				\draw[color=black] (8.6,3.5) node {$b^D_\rho([v_0, v_1])$};
			\end{scriptsize}
		\end{tikzpicture}
		\caption{Illustration of the second case.}
		\label{ill_second_case}
	\end{subfigure}
	\caption{Illustrations of the construction in the proof of Proposition \ref{expanded_polygon_prop}. The notation $b_\rho([v_0, v_1])$ and $b^D_\rho([v_0, v_1])$ is short for $b_\rho(\exp(i[v_0, v_1]))$ and $b^D_\rho(\exp(i[v_0, v_1]))$ respectively.}
	\label{illustrations_fig}
\end{figure}

	In the first case, consider the line through $P$ and $\lambda$. The choice of $\delta_{\rho}$ ensures that $P \neq \lambda$. Note that $P\lambda \perp AB$. Cut the line through $P$ and $\lambda$ at $\lambda$ so we end up with a ray (plotted in red in Figure \ref{ill_first_case}) from $\lambda$ not intersecting $AB$. Denote this ray by $r_\lambda$. Assume for contradiction that $r_\lambda$ intersects $b_\rho(e^{iv})$ for some $v \in [v_0, v_1]$, say at the point $C$, see Figure \ref{ill_first_case} for an illustration. As $ \abs{CP} \leq \norm{b_\rho - b_\rho^D}_\infty < \delta_\rho$, this is a contradiction since $\abs{P\lambda} \leq \abs{CP}$ and $\abs{\lambda C} \leq \abs{CP}$ (since $\lambda$ lies on the segment $CP$) but we have assumed that either $\lambda \notin (\spec T(b_\rho))_{\delta_{\rho}}$ or $\lambda \notin (\spec T(b_\rho^D))_{\delta_{\rho}}$. Hence neither $b_\rho(e^{iv})$ nor $b_\rho^D(e^{iv})$ intersect $r_\lambda$ for $v \in [v_0, v_1]$, so all these points belong to the same cut plane, namely $\mathbb{C} \setminus r_\lambda$. Therefore $\theta(v_1) = \theta_D(v_1)$ since $\theta(v_0) = \theta_D(v_0)$ and the property that $b_\rho(e^{iv})$ and $b_\rho^D(e^{iv})$ stay in the same cut plane for $v \in [v_0, v_1]$ implies that $\abs{\theta(v) - \theta_D(v)} < 2 \pi$ for $v \in [v_0, v_1]$, making $\theta(v) - \theta_D(v) = 2 \pi k$ impossible for a non zero integer $k$.
	
	We can handle the second case in a similar manner. Assume without loss of generality that $A$ is closer to $\lambda$ than $B$, and let $r_\lambda$ be the ray (plotted in red in Figure \ref{ill_second_case}) originating in $\lambda$ that lies on the line through $A$ and $\lambda$ and does not contain $A$. Assume as before that $r_\lambda$ intersects $b_\rho(e^{iv}), v \in [v_0, v_1]$ at some point $C$, see Figure \ref{ill_second_case} for an illustration. By definition of $\delta_\rho$, we see that $\abs{AC} < \delta_\rho$, and using a similar argument as in the first case gives that neither $b_\rho(e^{iv})$ nor $b_\rho^D(e^{iv})$ can intersect $r_\lambda$ for $v \in [v_0, v_1]$. This implies that $\theta(v_1) = \theta_D(v_1)$. Note that the case $A = B$ is also handled by this construction. 
	
	All in all we have shown that $\theta_D(v_j) - \theta_D(v_{j-1}) = \theta(v_j) - \theta(v_{j-1})$ for $j = 1$ and the argument can be continued inductively which shows it for all $1 \leq j \leq n$. From \eqref{winding_number_comparision} we now get that $\wind(b_\rho - \lambda) = \wind(b_\rho^D - \lambda)$ which means that $\lambda \notin (\spec T(b_\rho))_{\delta_{\rho}} \Rightarrow \lambda \notin \spec T(b_\rho^D)$ and $\lambda \notin (\spec T(b_\rho^D))_{\delta_{\rho}} \Rightarrow \lambda \notin \spec T(b_\rho)$, which is what we wanted to show.
	\end{proof}

Proposition \ref{expanded_polygon_prop} gives us a good bound on the error introduced by the polygon approximation. Next we bound $\norm{b_\rho - b_\rho^D}_\infty$ and use this to get a uniform bound for the error. 

\begin{lemma}
	Let $\rho_l$ and $\rho_h$ be bounds as in Proposition \ref{compact_rhos_theorem} and let $(v_j)_{j=0}^{n}$ be a partition of $[0, 2\pi]$ of granularity $\Delta$. Then there exists $C > 0$ depending on the symbol $b$ but not depending on $\Delta$ such that \begin{equation}
		d_H\left(\spec T(b_\rho), \spec T(b_\rho^D)\right) \leq C \Delta^2
		\label{disc_spectrum_error}
	\end{equation} for all $\rho \in [\rho_l, \rho_h]$. 
\label{uniform_error_bound_lemma}
\end{lemma}
\begin{proof}
	When comparing $b_\rho^D$ with $b_\rho$, it is convenient to write \begin{equation*}
		b(\rho e^{iv}) = x_\rho (v) + i y_\rho(v).
	\end{equation*}  Further, we see that $x_\rho(v)$ and $y_\rho(v)$ are smooth in both $\rho$ and $v$. In particular this implies that $x_\rho''(v)$ and $y_\rho''(v)$ are continuous on $[\rho_l, \rho_h] \times [0, 2 \pi]$, so \begin{equation*}
	C := \sup_{\rho \in [\rho_l, \rho_h], v \in [0, 2\pi]} \abs{x_\rho''(v)} + \abs{y_\rho''(v)} < \infty.
\end{equation*} By the mean value theorem for $x_\rho$ and $y_\rho$ separately we have for $v \in (v_j, v_{j+1})$ \begin{equation*}
	 \frac{b(\rho e^{iv_{j+1}}) - b(\rho e^{iv_{j}})}{v_{j+1} - v_j} = x_\rho'(v_{jx}) + i {y}_\rho'(v_{jy}),
\end{equation*} for some $v_{jx}, v_{jy} \in (v_j, v_{j+1})$. Using \eqref{polygon_definition} and the mean value theorem repeatedly, we get for $v \in (v_j, v_{j+1})$ \begin{equation}
\begin{split}
	\abs{b_\rho^D(e^{iv}) - b(\rho e^{iv})} &  = \abs{(v - v_j) ({x}_\rho'(v_{jx}) + i {y}_\rho'(v_{jy})) + b(\rho e^{iv_j}) - b(\rho e^{iv})} \\
	& = \abs{(v - v_j) ({x}_\rho'(v_{jx}) + i {y}_\rho'(v_{jy})) - (v - v_j) ({x}_\rho'(v_{jx}^*) + i {y}_\rho'(v_{jy}^*))} \\
	& \leq \abs{v - v_j} \left(\abs{{x}_\rho'(v_{jx}) - {x}_\rho'(v_{jx}^*)} + \abs{{y}_\rho'(v_{jy}) - {y}_\rho'(v_{jy}^*)}\right), \\ 
	& \leq \Delta^2 \bigl(\abs{{x}_\rho''(v_{jx}^{**})} + \abs{{y}_\rho''(v_{jy}^{**})}\bigr) \\ 
	& \leq C \Delta^2,
\end{split}
\label{nom_bound}
\end{equation} where $v_{jx}^*, v_{jy}^* \in (v_j, v)$ and  $v_{jx}^{**}, v_{jy}^{**}$ between $v_{jx}$ and $v_{jx}^*$, and $v_{jy}$ and $v_{jy}^*$ respectively. Now we have $\norm{b_\rho - b_\rho^D}_\infty \leq C\Delta^2$, which using Proposition \ref{expanded_polygon_prop} gives the desired result. 
\end{proof}

With this result we are ready to prove a stronger version of Theorem \ref{riemann_sum_like}. \begin{theorem}
	Let $\rho_l$ and $\rho_h$ be bounds as in Theorem \ref{compact_rhos_theorem}. Then \begin{equation}
		\lim_{\Delta_\rho, \Delta_v \to 0} \Bigg[\sup_{\substack{\boldsymbol{\rho} \in \mathcal{P}(\Delta_\rho) \\ \boldsymbol{v} \in \mathcal{P}(\Delta_v)}}  d_H\Biggl(\bigcap_{j=0}^n \spec T(b_{\rho_j}^D), \Lambda(b) \Biggr) \Bigg]  = 0,
		\label{polygon_intersection}
	\end{equation} where $\boldsymbol{\rho} = (\rho_j)_{j=0}^{n} $,  $\mathcal{P}(\Delta_\rho) = \mathcal{P}_{[\rho_l, \rho_h]}(\Delta_\rho)$ and $\boldsymbol{v} = (v_j)_{j=0}^m$, $\mathcal{P}(\Delta_v) = \mathcal{P}_{[0, 2\pi]}(\Delta_v)$.
	\label{best_result}
\end{theorem}
\begin{proof}
	Fix an $\epsilon > 0$. We can use the triangle inequality for $d_H$ to estimate \begin{equation}
		\begin{multlined}
			d_H \Biggl(\,\bigcap_{j=0}^n \spec T(b_{\rho_j}^D), \Lambda(b)\Biggr) \leq d_H \Biggl(\,\bigcap_{j=0}^n \spec T(b_{\rho_j}^D), \bigcap_{j=0}^n \spec T(b_{\rho_j})\Biggr) \\ + d_H \Biggl(\,\bigcap_{j=0}^n \spec T(b_{\rho_j}), \Lambda(b)\Biggr).	
		\end{multlined}
	\label{discrete_error}
	\end{equation} Let $\delta_\rho$ be such that $d_H \left(\cap_{j=0}^n \spec T(b_{\rho_j}), \Lambda(b)\right) < \epsilon/2$ for all partitions of $[\rho_l, \rho_h]$ with granularity less than $\delta_\rho$, which exists by Theorem \ref{riemann_sum_like}. Let $C > 0$ be as in Lemma \ref{uniform_error_bound_lemma} and choose $\delta_v$ small enough so that $d_H\left(\spec T(b_\rho), \spec T(b_\rho^D)\right) \leq C \delta_v^2 < \epsilon/2$ for all partitions of $[0, 2\pi]$ with granularity less than $\delta_v$ and all $\rho \in [\rho_l, \rho_h]$. Using this we can prove that $d_H \bigl(\cap_{j=0}^n \spec T(b_{\rho_j}^D), \cap_{j=0}^n \spec T(b_{\rho_j})\bigr) \leq \epsilon/2$, by noting that \begin{equation*}
		\begin{split}
			\Biggl(\,\bigcap_{j=0}^n \spec T(b_{\rho_j}^D)\Biggr)_{\epsilon/2} = \bigcap_{j=0}^n \left( \spec T(b_{\rho_j}^D)\right)_{\epsilon/2} & \supset  \bigcap_{j=0}^n  \spec T(b_{\rho_j}), \\
			\Biggl(\,\bigcap_{j=0}^n \spec T(b_{\rho_j})\Biggr)_{\epsilon/2} = \bigcap_{j=0}^n \left( \spec T(b_{\rho_j})\right)_{\epsilon/2} & \supset  \bigcap_{j=0}^n  \spec T(b_{\rho_j}^D).
		\end{split} 
	\end{equation*} This concludes our proof since we now have $d_H \bigl(\cap_{j=0}^n \spec T(b_{\rho_j}^D), \Lambda(b)\bigr) \leq \epsilon/2 + \epsilon/2 = \epsilon$.  
\end{proof}

\begin{remark}
The limit in \eqref{polygon_intersection} can actually be taken in any order, i.e. \begin{equation}
\lim_{\Delta_\rho, \Delta_v \to 0} \bigcap_{j=0}^n \spec T(b_{\rho_j}^D) = \lim_{\Delta_v \to 0} \lim_{\Delta_\rho \to 0} \bigcap_{j=0}^n \spec T(b_{\rho_j}^D) =  \lim_{\Delta_\rho \to 0} \lim_{\Delta_v \to 0} \bigcap_{j=0}^n \spec T(b_{\rho_j}^D) = \Lambda(b),
\label{limit_swap}
\end{equation} where the limits are to be understood as in \eqref{polygon_intersection}. This property is not the case for many infinite-dimensional spectral problems. The equalities in \eqref{limit_swap} can be seen by using the continuity of $\rho \to \spec T(b_\rho^D)$ and \eqref{disc_spectrum_error}.

\end{remark}

Thanks to Theorem \ref{best_result} we have the outline of an algorithm that we know will converge to $\Lambda(b)$. Given $b$ we sample $\rho$'s in $[\rho_l, \rho_h]$, and for each sampled $\rho$ construct an approximating polygon according to \eqref{polygon_definition}. Then, taking the intersection with respect to non-zero winding number of all the sampled polygons yields an approximating polygon for $\Lambda(b)$. By Theorem \ref{best_result} we know that the intersection of the polygons will approach $\Lambda(b)$ as we decrease the granularities. Pseudocode for the algorithm can be seen in Algorithm \ref{first_approach_alg} below.

\begin{algorithm}
	\caption{Basic approach to calculating $\Lambda(b)$ geometrically}\label{first_approach_alg}
	\begin{algorithmic}[5]
		\Procedure{CalcLimitSet}{$b, n, m$} \\ $b$: the symbol,\\ $n$: the number of sampled $\rho$'s, \\ $m$: number of sampled $v$'s.
		\State $\rho_l, \rho_h \gets$ bounds from Proposition \ref{compact_rhos_theorem}
		\State rhos $\gets$ sample $n+1$ $\rho$'s in $[\rho_l, \rho_h]$
		\State vs $\gets$ sample $m+1$ $v$'s in $[0, 2\pi]$ 
		\State $\Lambda \gets$ $b_{\text{rhos[1]}}^D(\text{vs})$
		\For{$i \gets 2$ to $n+1$}
			\State $\Lambda \gets \Lambda \cap b_{\text{rhos}[i]}^D(\text{vs})$
		\EndFor
		\State \textbf{return} $\Lambda$
		\EndProcedure
	\end{algorithmic}
\end{algorithm}

\begin{remark}
	It would be useful to quantify the error introduced by the polygon approximation. In fact, we see from the proof of Theorem \ref{best_result} that there exists $C > 0$ such that \begin{equation}
		d_H \Biggl(\,\bigcap_{j=0}^n \spec T(b_{\rho_j}^D), \Lambda(b)\Biggr) \leq C\Delta_v^2 +  d_H \Biggl(\,\bigcap_{j=0}^n \spec T(b_{\rho_j}), \Lambda(b)\Biggr).
		\label{error_separation}
	\end{equation} In other words, the polygon approximation gives a term that converges very fast to 0, so the main error to worry about is the second term on the right hand side of \eqref{error_separation}.
\label{different_errors_remark} 
\end{remark}

Now we have a way of producing an approximating polygon, but from just looking at the approximating polygon we do not get a lot of information about the size of the error, but this is likely something that is very useful to know in practical applications. 
For example, we would not be able to say for certain if the limit set lies in a specific region or not, since we do not know the size of the error. However, we could try to resolve this issue by using the following idea: if we slightly expand every polygon in \eqref{polygon_intersection} by an amount $\delta_\rho$, such that $(\spec T(b_{\rho_j}^D))_{\delta_{\rho_j}} \supset \spec T(b_{\rho_j})$, then the intersection of the expanded polygons would contain $\Lambda(b)$. The intersection of the expanded polygons would then yield an \textit{approximating superset} of $\Lambda(b)$, which we know that $\Lambda(b)$ is a subset of. And so if the approximating superset does not intersect a specific region, we know for certain that $\Lambda(b)$ does not do that either. 



\begin{corollary}
	Let $\rho_l$ and $\rho_h$ be bounds as in Theorem \ref{compact_rhos_theorem}, and let $(\rho_j)_{j=0}^{n}$ denote any partition of $[\rho_l, \rho_h]$ with granularity $\Delta_\rho$. Then there exists a family of constants $C_\rho$, $\rho \in [\rho_l, \rho_h]$ such that for all $\Delta_v > 0$ and $(v_j)_{j=0}^{m} \in \mathcal{P}_{[0, 2 \pi]}(\Delta_v)$ we have \begin{equation*}
		\begin{split}
				\bigcap_{j=0}^n \left( \spec T(b_{\rho_j}^D)\right)_{C_{\rho_j}\Delta_v^2} & \supset \Lambda(b).
		\end{split} 
	\end{equation*}  Further, we have that \begin{equation*}
		\lim_{\Delta_\rho, \Delta_v \to 0} \Bigg[\sup_{\substack{\boldsymbol{\rho} \in \mathcal{P}(\Delta_\rho) \\ \boldsymbol{v} \in \mathcal{P}(\Delta_v)}}  d_H\Biggl(\bigcap_{j=0}^n\left( \spec T(b_{\rho_j}^D)\right)_{C_{\rho_j}\Delta_v^2}, \Lambda(b) \Biggr) \Bigg]  = 0,
	\end{equation*}
	where $\boldsymbol{\rho} = (\rho_j)_{j=0}^{n} $,  $\mathcal{P}(\Delta_\rho) = \mathcal{P}_{[\rho_l, \rho_h]}(\Delta_\rho)$ and $\boldsymbol{v} = (v_j)_{j=0}^m$, $\mathcal{P}(\Delta_v) = \mathcal{P}_{[0, 2\pi]}(\Delta_v)$.
	\label{best_bound_result}
\end{corollary}

\begin{proof}
	Let $C_\rho := \norm{{x}_\rho''(v)}_\infty + \norm{{y}_\rho''(v)}_\infty + \eta$ for arbitrary $\eta \in (0, 1)$. From Proposition \ref{expanded_polygon_prop} and \eqref{nom_bound}, we see that $$\left(\spec T(b_{\rho_j}^D)\right)_{C_{\rho_j} \Delta_v^2} \supset \spec T(b_{\rho_j}),$$ which shows the first statement about the inclusion. 
	
	For the second statement, let $C := \sup_{\rho \in [\rho_l, \rho_h]} C_\rho$ and note that $$\bigcap_{j=0}^n \left( \spec T(b_{\rho_j}^D)\right)_{C_{\rho_j}\Delta_v^2} \subset \Biggl(\,\bigcap_{j=0}^n \spec T(b_{\rho_j}^D)\Biggr)_{C \Delta_v^2}. $$ Now let $\Delta_\rho$ and $\Delta_v$ be such that $d_H\bigl(\cap_{j=0}^n \spec T(b_{\rho_j}^D), \Lambda(b)\bigr) < \epsilon/2$ for partitions of granularities less than $\Delta_\rho$ and $\Delta_v$, respectively. Further, suppose $\Delta_v$ fulfills $C \Delta_v^2 < \epsilon/2$. Then $$d_H\Biggl(\,\bigcap_{j=0}^n \spec T(b_{\rho_j}^D), \biggl(\,\bigcap_{j=0}^n \spec T(b_{\rho_j}^D)\biggr)_{C_{}\Delta_v^2}\,\Biggr) < \epsilon/2, $$ and the triangle inequality for $d_H$ now gives what we want. 
\end{proof}

\begin{remark}
	Using Corollary \ref{best_bound_result} we can easily modify Algorithm \ref{first_approach_alg} to get a bound for $\Lambda(b)$. Simply use the constants given in the proof of Corollary \ref{best_bound_result}, expand the polygons we intersect with by the given amount, and compute the intersection in the same way as before. 
	\label{offset_bound_remark}
\end{remark}

We now have a way to produce an approximating superset of $\Lambda(b)$ and from the previous algebraic approach presented in Section \ref{previous_work_subsection}, we have a way of computing an approximating subset of $\Lambda(b)$ in the sense that the set of points produced by the algebraic algorithm also will be a subset of $\Lambda(b)$ and converge to $\Lambda(b)$ in the Hausdorff metric. Let $\Lambda^{sup}$ denote the approximating superset, and let $\Lambda_{sub}$ denote the approximating subset. We have $\Lambda_{sub} \subset \Lambda(b) \subset \Lambda^{sup}$, and from the definition of $d_H$ we easily see that $d_H(\Lambda_{sub}, \Lambda(b)) \leq d_H(\Lambda_{sub}, \Lambda^{sup})$ and $d_H(\Lambda^{sup}, \Lambda(b)) \leq d_H(\Lambda_{sub}, \Lambda^{sup})$. So this means that we have full explicit error control in the Hausdorff metric. We summarize this argument in Algorithm \ref{error_control} below.

\begin{algorithm}
	\caption{Compute an upper bound on $d_H(\Lambda^{sup}, \Lambda(b))$ and $d_H(\Lambda_{sub}, \Lambda(b))$}\label{error_control}
	\begin{algorithmic}[5]
		\Procedure{CalcErrorBound}{$b, n, m$} \\ $b$: the symbol,\\ $n$: the number of sampled $\rho$'s, \\ $m$: number of sampled $v$'s.
		\State $\Lambda^{sup} \gets$ compute approximating superset
		\State $\Lambda_{sub} \gets$ compute approximating subset using Section \ref{previous_work_subsection}
		\State $d\gets d_H(\Lambda_{sub}, \Lambda^{sup})$ 
		\State \textbf{return} $d$
		\EndProcedure
	\end{algorithmic}
\end{algorithm}

\begin{remark}
	When constructing the approximating superset, it is possible that we lose interior regions of the approximating polygon that have zero winding number (i.e., what appears from the approximating polygon to be loops of $\Lambda(b)$), this is just a matter of choosing too large $\Delta_v$ or too large $C_\rho$'s. However, Algorithm \ref{error_control} gives us a bound on how large holes of $\Lambda(b)$ the approximating superset can cover.
\end{remark}

\section{Implementation details}
\label{implementation_details_section}
The main operation we need in Algorithm \ref{first_approach_alg} is polygon intersection, which is a common problem in the field of computer graphics. Hence, it has been thoroughly researched, and there exist efficient implementations. For this work, we use the Python library pyclipper \cite{pyclipper}, which is a Python wrapper library for the C$++$ library Clipper2 \cite{clipper2}. The library is based on Vatti's clipping algorithm, which is described in his article \cite{vatti_clipping}. The processing time for intersecting two polygons using Vatti's algorithm scales linearly in the number of vertices of the polygons being intersected. 

To produce the bound for $\Lambda(b)$ as described in Remark \ref{offset_bound_remark}, we also need to be able to expand polygons. This problem is called polygon offsetting in the literature, and it is a well researched problem. In fact, Clipper2 contains an implementation of the offsetting algorithm as described in \cite{offsetting_article}. It should be noted that we do not want to compute $(\spec T(b_\rho^D))_{C_\rho \Delta_v^2}$ exactly, since the boundary then contains circular arcs, which are more difficult to intersect with than polygons. Instead, we bound $(\spec T(b_\rho^D))_{C_\rho \Delta_v^2}$ with a polygon ${P}_\rho$ such that \begin{equation}
	\left((\spec T(b_\rho^D)\right)_{C_\rho \Delta_v^2} \subset {P}_\rho \subset \left(\spec T(b_\rho^D)\right)_{2 C_\rho \Delta_v^2}
	\label{polygon_bound}
\end{equation} The constant 2 here is chosen quite arbitrarily, and any constant greater than one would work in theory. For practical reasons we want the chosen constant to be ``large enough'' so that $P_\rho$ does not contain too many vertices, and ``small enough'' so that $P_\rho$ is close to $\Lambda(b)$. Because of the way $C_\rho$ is chosen in the proof of Corollary \ref{best_bound_result}, the conclusions hold also for $2C_\rho$, and because of \eqref{polygon_bound}, they then also hold for ${P}_\rho$. Clipper2 contains an implementation for choosing such ${P}_\rho$. The time complexity for these offset polygons is a bit tricky to analyze, as the offsetting algorithm has time complexity $O((m+k)\log m)$ where $m$ is the number of vertices of the input polygon and $k$ the number of self-intersections in the raw offset curve \cite{offsetting_article}. Intuition tells us that for sufficiently regular polygons, such as the ones we work with, $k$ should be of the same order as $m$, so in the following analysis of time complexity we assume that producing ${P}_\rho$ takes $O(m \log m)$ time and that the number of vertices for ${P}_\rho$ is $O(m)$. 

Using the above implementations of offsetting and polygon intersection we can compute the approximating polygon and approximating superset of $\Lambda(b)$ given by Algorithm \ref{first_approach_alg}. We do this for the symbol presented in Example \ref{main_example}. The results can be seen in Figure \ref{first_alg_all_results}. If we look at the bigger picture as in Figure \ref{first_alg_est}, we see that the results are very similar to those of the algebraic approach. And if we zoom in, we see that our geometric approach struggles at the origin, which also is true for the algebraic algorithm if one zooms even further, but it can't be clearly seen in the figures. 
It took roughly 3 minutes to run the algorithm for finding the approximating polygon and the approximating superset.

\begin{figure}
	\centering
	\begin{subfigure}[t]{0.49\textwidth}
		\centering
		\includegraphics[width=\textwidth]{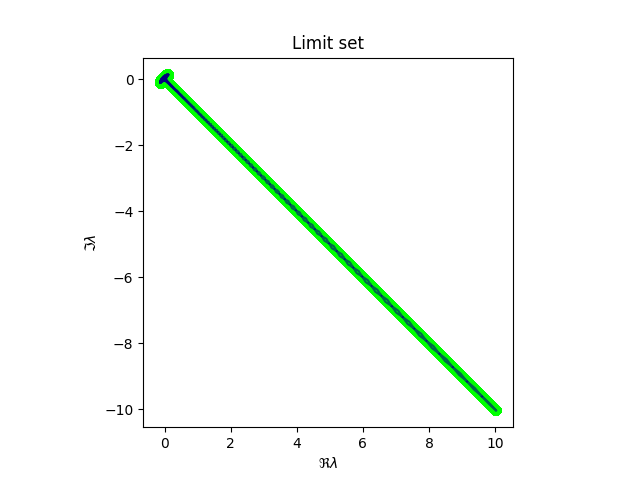}
		\caption{An approximating polygon for $\Lambda(b)$ computed using Algorithm \ref{first_approach_alg} with $n = 1250$ and $m = 1000$.}
		\label{first_alg_est}
	\end{subfigure}
	\hfill
	\begin{subfigure}[t]{0.49\textwidth}
		\centering
		\includegraphics[width=\textwidth]{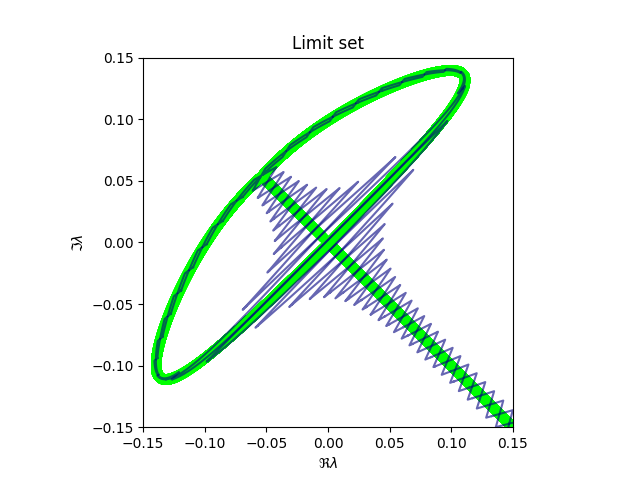}
		\caption{A zoomed in version of \ref{first_alg_est}.}
		\label{first_alg_est_zoom}
	\end{subfigure}
	
	\begin{subfigure}[t]{0.49\textwidth}
		\centering
		\includegraphics[width=\textwidth]{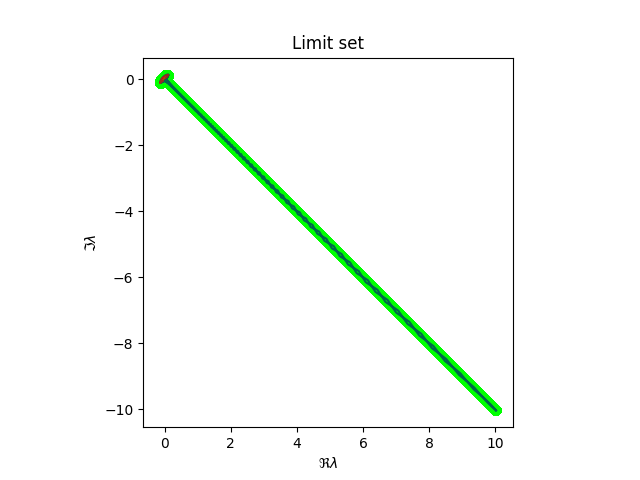}
		\caption{An approximating superset of $\Lambda(b)$ computed using Algorithm \ref{first_approach_alg} with $n = 1250$ and $m = 1000$.}
		\label{first_bound_est}
	\end{subfigure}
	\hfill
	\begin{subfigure}[t]{0.49\textwidth}
		\centering
		\includegraphics[width=\textwidth]{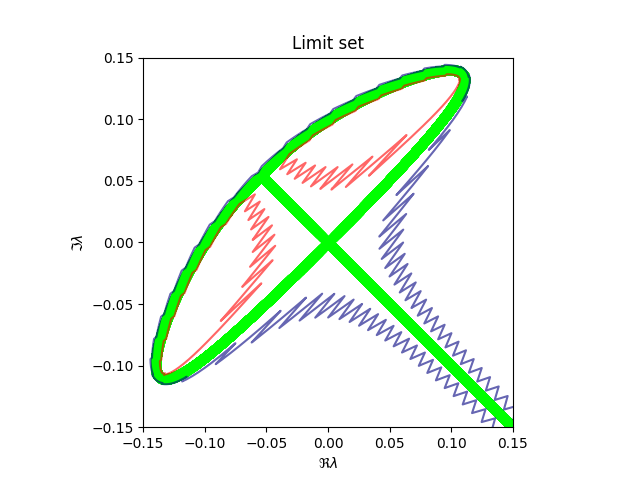}
		\caption{A zoomed in version of \ref{first_bound_est}. Blue polygons are positively oriented and the red ones are negatively oriented. That is, the approximating superset is the region outside the red polygons and inside the blue polygon.}
		\label{first_alg_bound_zoom}
	\end{subfigure}

	\caption{A comparison between Algorithm \ref{first_approach_alg} and the algebraic approach presented in \ref{previous_work_subsection}. The symbol $b$ being investigated is presented in \eqref{main_example_eq}. The lime dots in all of the subfigures are the results from the algebraic approach, these are the same results as presented in Figure \ref{algebraic_alg_all_results}.}
	\label{first_alg_all_results}
\end{figure}


Besides finding efficient polygon intersection and offsetting algorithms, the most important issue in the implementation is how to choose $(v_j)_{j=0}^{m}$ and $(\rho_j)_{j=0}^{n}$ (i.e., where to choose the vertices of $b_{\rho_j}^D$ on $b_{\rho_j}(\mathbb{T})$ and what $\rho$'s to choose for the intersections). We know from Theorem \ref{best_result} that if we just choose the granularities to be finer and finer, then the intersection will approach $\Lambda(b)$. But we should be able to get a more computationally efficient algorithm if the partitions are chosen in a good way. 

For simplicity, the only investigated partition for $v$ is the uniform one, that is, $v_j = 2\pi j/m$. It is quite difficult to imagine some strategy working much better. Choosing $\rho$'s on the other hand is more interesting. A natural approach is to first calculate $[\rho_l, \rho_h]$ using the estimate from Proposition \ref{compact_rhos_theorem}. Then, a naive solution is to sample $\rho$'s uniformly in this interval.
There are two major points of critique for this method of sampling $\rho$'s. First of all, the bounds $[\rho_l, \rho_h]$ are not strict. For the $\rho$'s that are close to the end points, $\spec T(b_\rho)$ probably contains the entirety of the currently intersected polygon. Secondly, from the proof of Theorem \ref{riemann_sum_like} we see that ideally, we want to choose more $\rho$'s where $d_H(\spec T(b_\rho), \spec T(b_{\tilde{\rho}}))$ grows rapidly when moving $\tilde{\rho}$ away from $\rho$.

To get an intuition of good heuristics for choosing stricter bounds for $\rho$, it is helpful to plot $b_\rho(\mathbb{T})$ and let $\rho$ vary. An insight from experimenting is that $\Lambda(b)$ is formed by the points where $b_\rho(\mathbb{T})$ intersects itself, and when $\rho$ is varied, this intersection moves along $\Lambda(b)$ in a continuous manner. This would suggest that we only want to find the intervals for which $b_\rho(\mathbb{T})$ self-intersect, and where these intersections are a part of $\Lambda(b)$. As these intersections move continuously with $\rho$, there should be a finite number of interesting intervals for $\rho$. Actually, one can see from \eqref{graph_intersect_property} that all $\lambda \in \Lambda(b)$, except for a finite number, must arise from a self-intersection of $b_\rho(\mathbb{T})$, since \eqref{graph_intersect_property} exactly describes this type of intersection.

To find the interesting intervals of $\rho$'s analytically is not obvious. But we can use the fact that when we intersect our partial polygon with $\spec T(b_\rho)$ for a good $\rho$, the area of the partial polygon decreases. So we could sweep over $[\rho_l, \rho_h]$ and calculate the decrease in area when intersecting with different $\spec T(b_\rho)$. Also, note that the area of a polygon can be calculated in linear time with respect to the number of vertices. Pseudocode for calculating the approximating polygon and approximating superset of $\Lambda(b)$ with the sweep improvement can be seen in Algorithm \ref{area_sweep_alg}. How to choose the intervals of $\rho$'s where the area is reduced the most deserves some attention. If some $\rho$ is in both \code{rhos} and \code{sweeprhos} in Algorithm \ref{area_sweep_alg}, the area reduced will be 0. To avoid missing out on good $\rho$'s because of this, a moving average filter is used for \code{areareduce}, and then all the good intervals are selected as the ones where \code{areareduce}[$\rho$] $\geq (\max_\rho\text{\code{areareduce}}[\rho])/10^6$. 

The problem of sampling the same $\rho$ also applies when we sample \code{rhos} for the next intersection-run. Typically, we will find the exact same interval of good $\rho$'s that reduce area the most in each iteration. Therefore, a randomly chosen $\epsilon$ with $0 \leq \epsilon \leq \text{sweeprhos}[i+1] - \text{sweeprhos}[i]$ is added to the left boundaries of the good intervals and subtracted from the right boundaries. 

A significant optimization can be done when we calculate \code{areareduce}. The important observation is that \code{areareduce}$[\rho]$ is non-increasing. When updating \code{areareduce} we begin by calculating $\Lambda_{s}$ for the intervals of $\rho$ that we just intersected with. Then, if \code{areareduce} for the $\rho$ that did not reduce area the most in the previous iteration still are smaller than the threshold, we don't need to update those values. Practically, this drastically reduces the number of intersections done in total. Most polygon intersections used for calculating $\Lambda_{s}$ are done in the first iteration, when we have no prior values in \code{areareduce}.

The second major point of critique to the naive method means that we ideally want to sample $\rho$'s denser where $d_H(\spec T(b_\rho), \spec T(b_{\tilde{\rho}}))$ grows rapidly when moving $\tilde{\rho}$ away from $\rho$. We tried a strategy of doing this. The idea behind our strategy is the observation that $b$ can be seen as a sum of two polynomials, one in $z$ and one in $1/z$. The standard approach if nothing better is known is to sample uniformly, but because of $1/z$, it would make sense to sample $\rho_j$ uniformly for $\rho \geq 1$ and sampling $\rho < 1$ in such a way that $1/\rho_j$ becomes uniformly distributed. 

We note that besides the above mentioned improvements, Algorithm \ref{area_sweep_alg} also contains the details on how to produce the approximating superset, which is denoted $\Lambda_b$ in the pseudocode. From the proof of Corollary \ref{best_bound_result}, we see that we should put  $C_\rho = \sup_{v \in [0, 2\pi]}  \norm{{x}_\rho''(v)}_\infty + \norm{{y}_\rho''(v)}_\infty + \eta$. And reasonably, to get an efficient algorithm in practice, we should scale $\eta$ with $\sup_{v \in [0, 2\pi]}  \norm{{x}_\rho''(v)}_\infty + \norm{{y}_\rho''(v)}_\infty$. So we chose $\eta := 0.2 \cdot \sup_{v \in [0, 2\pi]}  \norm{{x}_\rho''(v)}_\infty + \norm{{y}_\rho''(v)}_\infty$, since we need something positive, but small enough to not worsen the result significantly. This is why we have the constant 1.2 in Algorithm \ref{area_sweep_alg}. 


\begin{algorithm}
	\caption{Improved $\Lambda(b)$--calculator using area sweeps}\label{area_sweep_alg}
	\begin{algorithmic}[5]
		\Procedure{CalcLimitSet}{$b, n, m, l, \#sweeps$} \\ $b$: the symbol,\\ $n$: the number of sampled $\rho$'s, \\ $m$: number of sampled $v$'s, \\ $l$: the number of sampled $\rho$'s for area sweeping, \\ $\#sweeps$: how many times new $\rho$'s should be generated.  
		\State $\rho_l, \rho_h \gets$ bounds from Theorem \ref{compact_rhos_theorem}
		\State sweeprhos $\gets$ sample $l$ points for area sweeping in $[\rho_l, \rho_h]$.
		\State rhos $\gets$ sample $n/\#sweeps$ points in $[\rho_l, \rho_h]$
		\State vs $\gets$ sample $m+1$ points in $[0, 2\pi]$ 
		\State $\Delta_v \gets \max_{1 \leq i \leq m} (\text{vs}[i+1] - \text{vs}[i])$
		\State $\rho \gets \text{rhos}[1]$
		\State $\Lambda \gets$ $b_{\rho}^D(\text{vs})$
		\State $C_\rho \gets 1.2 \cdot \sup_{v \in [0, 2\pi]}  (\norm{{x}_\rho''(v)}_\infty + \norm{{y}_\rho''(v)}_\infty)$
		\State $\Lambda_b \gets$ $(b_{\rho}^D(\text{vs}))_{C_\rho \Delta_v^2}$
		\For{$i \gets 1$ to $\#sweeps$}
			\For{each $\rho$ in rhos}
			\State $C_\rho \gets 1.2 \cdot \sup_{v \in [0, 2\pi]}  (\norm{{x}_\rho''(v)}_\infty + \norm{{y}_\rho''(v)}_\infty)$ 
			\State $\Lambda \gets \Lambda \cap b_{\rho}^D(\text{vs})$
			\State $\Lambda_b \gets \Lambda_b \cap (b_{\rho}^D(\text{vs}))_{C_\rho \Delta_v^2}$
			\EndFor
			\For{each $\rho$ in sweeprhos}
			\State $\Lambda_{s} \gets \Lambda \cap b_{\rho}^D(\text{vs})$
			\State areareduce[$\rho$] $\gets$ Area($\Lambda$) - Area($\Lambda_{s}$)
			\EndFor
			\State rhos $\gets$ sample $n/\#sweeps$ points in the intervals that reduces area the most
		\EndFor
		\State \textbf{return} $\Lambda, \Lambda_b$
		\EndProcedure
	\end{algorithmic}
\end{algorithm}

The difficult step of Algorithm \ref{error_control} is to compute the Hausdorff distance $d_H(\Lambda_{sub}, \Lambda^{sup})$. Since $\Lambda_{sub} \subset \Lambda^{sup}$, we want to find the smallest $r$ such that $\left(\Lambda_{sub}\right)_r \supset \Lambda^{sup}$. The idea behind our approach to do this is to binary search over $r$. We form regular $n$-gons of radius $r$ around each point in $\Lambda_{sub}$ and then compute the union of all $n$-gons. Then we intersect the union with $\Lambda^{sup}$, and check if the intersection is $\Lambda^{sup}$. If it is, it means that we can choose $r$ smaller, and if the intersection is not $\Lambda^{sup}$, it means that we must choose $r$ larger. Because of the approximation with a regular $n$-gon, the result is not exact, but one can show through a geometric argument comparing the radii of the circumscribed and inscribed circle of a regular $n$-gon that the output from our approach, $r^*$, satisfies $r^* \cos(\pi / n) \leq d_H(\Lambda_{sub}, \Lambda^{sup}) \leq r^*$. In our case, we chose $n = 20$ so that $0.98 r^* \leq d_H(\Lambda_{sub}, \Lambda^{sup}) \leq r^*$. 


\subsection{Testing suggested improvements}

Regarding the suggested improvement about sampling $\rho$'s denser where the Hausdorff metric varies rapidly, we found that the ideal strategy for sampling $\rho$ is highly dependent on the symbol. So in practical applications where high performance is required, the user should fine tune the sampling strategy for the specific symbol. 

To investigate the other suggested improvements (i.e., using Algorithm \ref{area_sweep_alg} instead of Algorithm \ref{first_approach_alg}), we again compute an approximating polygon and approximating superset of $\Lambda(b)$ for the symbol given in Example \ref{main_example}. To get a fair comparison, we run Algorithm \ref{area_sweep_alg} with $n= 1000$, $m = 1000$, $l = 250$ and $\#sweeps = 2$, which means that we use the same number of polygon intersections for both algorithms. 

The approximating polygon and approximating superset of $\Lambda(b)$ computed with Algorithm \ref{area_sweep_alg} can be seen in Figure \ref{second_alg_all_results}. We cannot really see any big difference compared to Figure \ref{first_alg_all_results} in the zoomed out pictures. But if we look at Figures \ref{second_alg_est_zoom} and \ref{second_alg_bound_zoom}, we see that the approximating polygon looks more like the type of set we are searching for (a connected finite union of analytic arcs), and the approximating superset in \ref{second_alg_bound_zoom} is tighter than the one in \ref{first_alg_bound_zoom}. These properties should generalize since it is never good to ``waste'' polygon intersections. It took roughly 3 minutes to run the algorithm for finding the approximating polygon and approximating superset.

We also wish to use Algorithm \ref{error_control} to see the difference between Algorithms \ref{first_approach_alg} and \ref{area_sweep_alg}. However, the running time of our implementation of Algorithm \ref{error_control} is quadratic in the number of sampled $\varphi$'s, so we cannot sample $2\cdot 10^5$ $\varphi$'s as in Figures \ref{algebraic_alg_all_results}, \ref{first_alg_all_results}, \ref{second_alg_all_results}. Instead we sample $2 \cdot 10^3$ $\varphi$'s and run Algorithms \ref{first_approach_alg} and \ref{area_sweep_alg} with the same parameters as in Figures \ref{first_alg_all_results} and \ref{second_alg_all_results}. We then compute $d_H(\Lambda_{sub}, \Lambda^{sup}) \approx 0.045$ for Algorithm \ref{first_approach_alg} and $d_H(\Lambda_{sub}, \Lambda^{sup}) \approx 0.035$ for Algorithm \ref{area_sweep_alg}. This is a notable difference, and it is also what one sees in Figures \ref{first_alg_bound_zoom} and \ref{second_alg_bound_zoom}.



\begin{figure}
	\centering
		\begin{subfigure}[t]{0.49\textwidth}
		\centering
		\includegraphics[width=\textwidth]{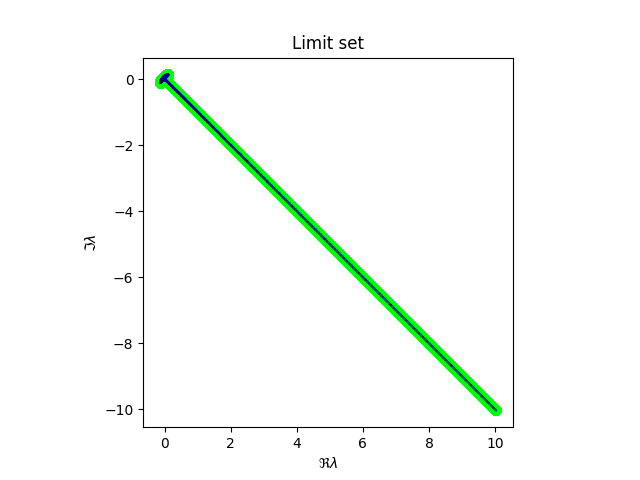}
		\caption{An approximating polygon for $\Lambda(b)$ computed using Algorithm \ref{area_sweep_alg} with $n = 1000$, $m = 1000$, $l = 250$, $\#sweeps=2$.}
		\label{second_alg_est}
	\end{subfigure}
	\hfill
	\begin{subfigure}[t]{0.49\textwidth}
		\centering
		\includegraphics[width=\textwidth]{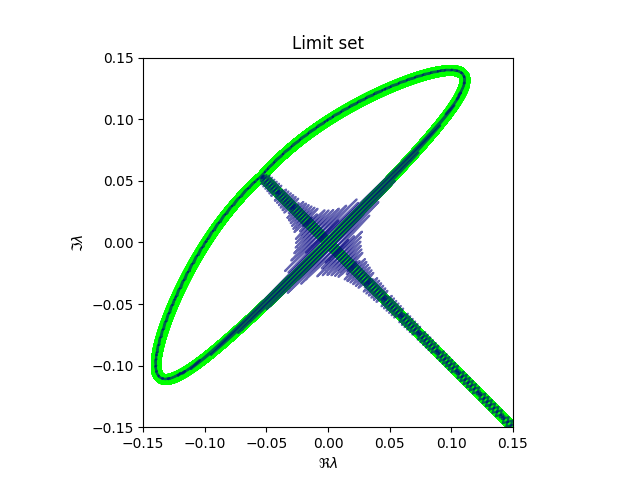}
		\caption{A zoomed in version of \ref{second_alg_est}.}
		\label{second_alg_est_zoom}
	\end{subfigure}
	
	\begin{subfigure}[t]{0.49\textwidth}
		\centering
		\includegraphics[width=\textwidth]{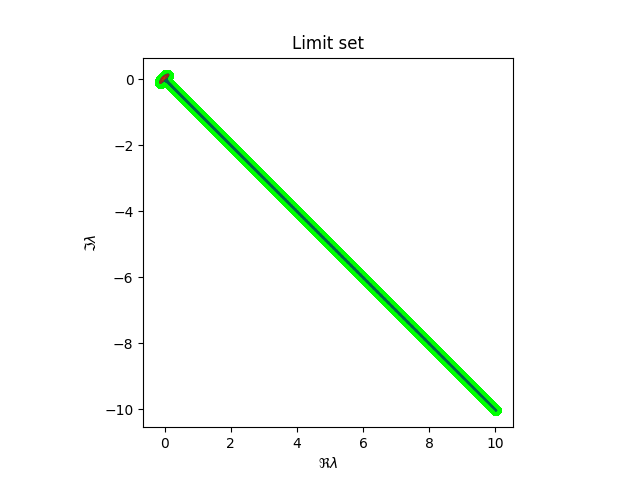}
		\caption{An approximating superset of $\Lambda(b)$ computed using Algorithm \ref{area_sweep_alg} with $n = 1000$, $m = 1000$, $l = 250$, $\#sweeps=2$.}
		\label{second_alg_bound}
	\end{subfigure}
	\hfill
	\begin{subfigure}[t]{0.49\textwidth}
		\centering
		\includegraphics[width=\textwidth]{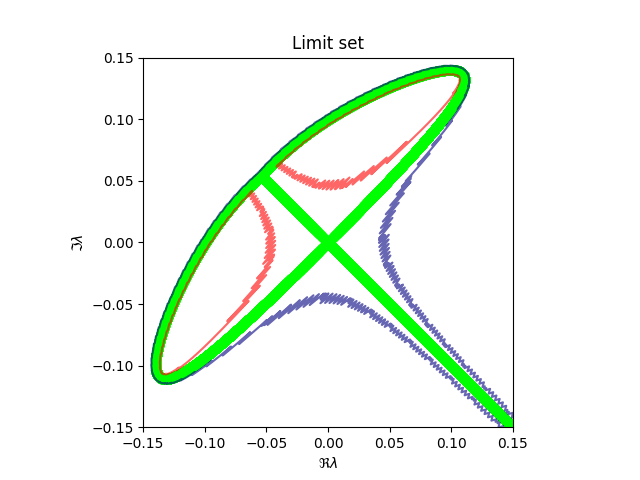}
		\caption{A zoomed in version of \ref{second_alg_bound}. Blue polygons are positively oriented and the red ones are negatively oriented. That is, the approximating superset is the region outside the red polygons and inside the blue polygon.}
		\label{second_alg_bound_zoom}
	\end{subfigure}

	\caption{The results of Algorithm \ref{area_sweep_alg}. The symbol $b$ being investigated is presented in \eqref{main_example_eq}. The lime dots in all of the subfigures are the results from the algebraic approach, these are the same results presented in Figure \ref{algebraic_alg_all_results}.}
	\label{second_alg_all_results}
\end{figure}

\section{Time complexity and convergence speed}
\label{time_complexity_conv_speed_section}
Analyzing the time complexity and convergence speed of Algorithm \ref{area_sweep_alg} rigorously seems difficult, but we will provide an argument and an example to back up that the running time for Algorithm \ref{area_sweep_alg} on average probably is $O(n^2 + nm\log m)$. Using that estimate, we then argue that the distance from the approximating polygon to $\Lambda(b)$ for most of $\Lambda(b)$ decreases as $O(1/\sqrt{k})$ where $k$ is the number of elementary operations. 

\subsection{Time complexity}

We wish to analyze the time complexity of Algorithm \ref{area_sweep_alg}, and to simplify that analysis we begin by looking at only the part of the Algorithm that produces the approximating polygon for $\Lambda(b)$. The most computationally complex parts of Algorithm \ref{area_sweep_alg} are the polygon intersections which we know are linear in the number of vertices of the incoming polygons. So what we really want to analyze is how the number of vertices in the approximating polygon $\Lambda$ grows. Let $w(i)$ be the number of vertices of $\Lambda$ after $i$ intersections. We will now argue that $w(i)$ should increase at most linearly. To see this, consider Figure \ref{nbr_vert_illustration}. When we intersect $\Lambda$ with $b_\rho^D(\text{vs})$, the typical situation is described there. Specifically, Figure \ref{intersection_illustration} is zoomed in where $b_\rho^D(\text{vs})$ intersects itself and the intersection is part of $\Lambda(b)$, which happens a bounded number of times. In Figure \ref{nbr_vert_illustration} we see that 3 vertices are added to $\Lambda$ for the specific self-intersection of $b_\rho^D(\text{vs})$. In Figure \ref{after_intersection} we see what $\Lambda$ looks like after it has been intersected. Since the number of self-intersections of $b_\rho^D(\text{vs})$ is bounded, $w(i)$ should increase at most linearly.

\begin{figure}[h!]
	\centering
	\definecolor{ttttff}{rgb}{0.2,0.2,1}
	\definecolor{ududff}{rgb}{0.30196078431372547,0.30196078431372547,1}
	\definecolor{qqccqq}{rgb}{0,0.8,0}
	\definecolor{ffqqqq}{rgb}{1,0,0}
	\begin{subfigure}[t]{0.49\textwidth}
		\centering
		\resizebox{\textwidth}{!}{
		\begin{tikzpicture}[line cap=round,line join=round,>=triangle 45,x=1cm,y=1cm, scale=1.0]
			\clip(12.43717510195849,0.5739278647084349) rectangle (16.574269646185698,3.961813591297184);
			\draw [samples=1000,rotate around={-46.7856210387764:(13.62145901449673,3.3052457502228494)},xshift=13.62145901449673cm,yshift=3.3052457502228494cm,line width=1pt,color=ffqqqq,domain=-56.53538808862847:56.53538808862847)] plot (\x,{(\x)^2/2/14.133847022157118});
			\draw [line width=1pt] (11.923939342293599,5.507708536294194)-- (14.419691767303828,5.358376344160063);
			\draw [line width=1pt] (14.419691767303828,5.358376344160063)-- (13.554695999991635,1.3986179426864822);
			\draw [line width=1pt] (13.554695999991635,1.3986179426864822)-- (17.886082212606766,1.693357241178043);
			\draw [line width=1pt] (15.669130097865885,0.033846843149471235)-- (15.584683182545474,2.926438758168702);
			\draw [line width=1pt] (11.923939342293599,5.507708536294194)-- (11.650541365309529,3.1230249391194955);
			\draw [line width=1pt] (11.650541365309529,3.1230249391194955)-- (15.584683182545474,2.926438758168702);
			\draw [line width=1pt,color=ttttff,domain=12.43717510195849:16.574269646185698] plot(\x,{(-1.3001905643435734-0.44661306873988815*\x)/-3.512948346207315});
			\draw [line width=1pt,color=ttttff,domain=12.43717510195849:16.574269646185698] plot(\x,{(--35.364148887537475-2.4763332171978174*\x)/-0.4998855171300107});
			\begin{scriptsize}
				\draw[color=ffqqqq] (15.046614593635368,43.247765714237595) node {Limit set};
				\draw [fill=qqccqq] (14.733687950657702,2.2432576221482705) circle (2pt);
				\draw [fill=ududff] (18.246636296865017,2.6898706908881587) circle (2pt);
				\draw [fill=qqccqq] (15.601407519500473,2.353573761542953) circle (2pt);
				\draw [fill=qqccqq] (14.878719546951764,2.96171524322155) circle (2pt);
				\draw [fill=qqccqq] (14.57723051980979,1.4681986940350553) circle (2pt);
				\draw [fill=qqccqq] (13.710797086666666,2.1132140283544016) circle (2pt);
				\draw [fill=ffqqqq] (15.584683182545474,2.9264387581687035) circle (2pt);
				\draw [fill=ffqqqq] (13.554695999991635,1.3986179426864829) circle (2pt);
			\end{scriptsize}
		\end{tikzpicture}
	}
	\caption{The red line is the true $\Lambda(b)$, the black lines are $\Lambda$ prior to being intersected, and the blue lines are $b_\rho^D(\text{vs})$. The figure is zoomed in where $b_\rho^D(\text{vs})$ intersects itself. The green points mark the new vertices in $\Lambda$, and the red ones are the vertices being removed from $\Lambda$.}
	\label{nbr_vert_illustration}
	\end{subfigure}
	\hfill
	\begin{subfigure}[t]{0.49\textwidth}
		\centering
		\resizebox{\textwidth}{!}{
	\begin{tikzpicture}[line cap=round,line join=round,>=triangle 45,x=1cm,y=1cm, scale=1.0]
		\clip(12.43717510195849,0.5739278647084349) rectangle (16.574269646185698,3.961813591297184);
		\draw [samples=1000,rotate around={-46.7856210387764:(13.62145901449673,3.3052457502228494)},xshift=13.62145901449673cm,yshift=3.3052457502228494cm,line width=1pt,color=ffqqqq,domain=-56.53538808862847:56.53538808862847)] plot (\x,{(\x)^2/2/14.133847022157118});
		\draw [line width=1pt] (11.923939342293599,5.507708536294194)-- (14.419691767303828,5.358376344160063);
		\draw [line width=1pt] (14.419691767303828,5.358376344160063)-- (13.7107970866667, 2.11321402835440);
		\draw [line width=1pt] (14.5772305198098, 1.46819869403506) -- (17.886082212606766,1.693357241178043);
		\draw [line width=1pt] (15.669130097865885,0.033846843149471235)-- (15.6014075195005, 2.35357376154295);
		\draw [line width=1pt] (11.650541365309529,3.1230249391194955)-- (14.8787195469518, 2.96171524322155);
		\draw [line width=1pt] 
		(13.7107970866667, 2.11321402835440) -- (15.6014075195005, 2.35357376154295);
		\draw [line width=1pt] 
		(14.8787195469518, 2.96171524322155) -- (14.5772305198098, 1.46819869403506);
		\begin{scriptsize}
			\draw [fill=qqccqq] (14.733687950657702,2.2432576221482705) circle (2pt);
			\draw [fill=qqccqq] (15.601407519500473,2.353573761542953) circle (2pt);
			\draw [fill=qqccqq] (14.878719546951764,2.96171524322155) circle (2pt);
			\draw [fill=qqccqq] (14.57723051980979,1.4681986940350553) circle (2pt);
			\draw [fill=qqccqq] (13.710797086666666,2.1132140283544016) circle (2pt);
		\end{scriptsize}
	
	\end{tikzpicture}
	}
	\caption{An illustration of $\Lambda$ from Algorithm \ref{area_sweep_alg} after it has been intersected as depicted in Figure \ref{nbr_vert_illustration}. The red line is $\Lambda(b)$ and the black lines are $\Lambda$. The green dots are the new vertices of $\Lambda$.}
	\label{after_intersection}
	\end{subfigure}
	
	\caption{A probable illustration of how it looks when $\Lambda$ is intersected with $b_\rho^D(\text{vs})$ in Algorithm \ref{area_sweep_alg}. }
	\label{intersection_illustration}
\end{figure}

To test our hypothesis about $w(i)$ growing linearly, we generate a random Laurent polynomial with $r=s=5$ and coefficients chosen uniformly in $[-10, 10] \times [-10i, 10i]$ and rounded to three decimals. We get the Laurent polynomial \begin{equation}
	\begin{multlined}
	b(t) = (-0.304+5.958i)t^{5} + (-6.954+8.098i)t^{4} + (-0.016-6.911i)t^{3} \\ + (-7.017-8.355i)t^{2} + (1.766+9.48i)t^{1} + (-9.161-6.354i) \\ + (-3.186+6.321i)t^{-1} + (-5.482+0.833i)t^{-2} + (-8.159-5.271i)t^{-3} \\ + (4.942+4.362i)t^{-4} + (-7.786-5.305i)t^{-5}.
	\end{multlined}
\label{random_pol}
\end{equation} We then run Algorithm \ref{area_sweep_alg} with $m = 500$, $n = 3000$,  $l = 500$, $\#sweeps=6$, $\rho$ sampled uniformly, and keep track of $w(i)$, and plot it in Figure \ref{number_vertices_growth}. The resulting approximating polygon can be seen in Figure \ref{number_vertices_growth_limit_set}. Figure \ref{number_vertices_growth} looks like it should according to our argument: $w(i)$ appears to grow piecewise linearly, with the slope changing. These changes correspond to the alterations of the amount of self-intersections. The algorithm finds the same interval of good $\rho$'s after the first 2 intersection-runs, which explains the repetition of the pattern. Also note that the repetitive pattern we see is exactly what should happen if our reasoning in Figure \ref{nbr_vert_illustration} is correct, since we intersect over the same interval of $\rho$'s each time. 
\begin{figure}[H]
	\centering
\begin{subfigure}[t]{0.49\textwidth}
	\centering
	\includegraphics[width=0.95\textwidth]{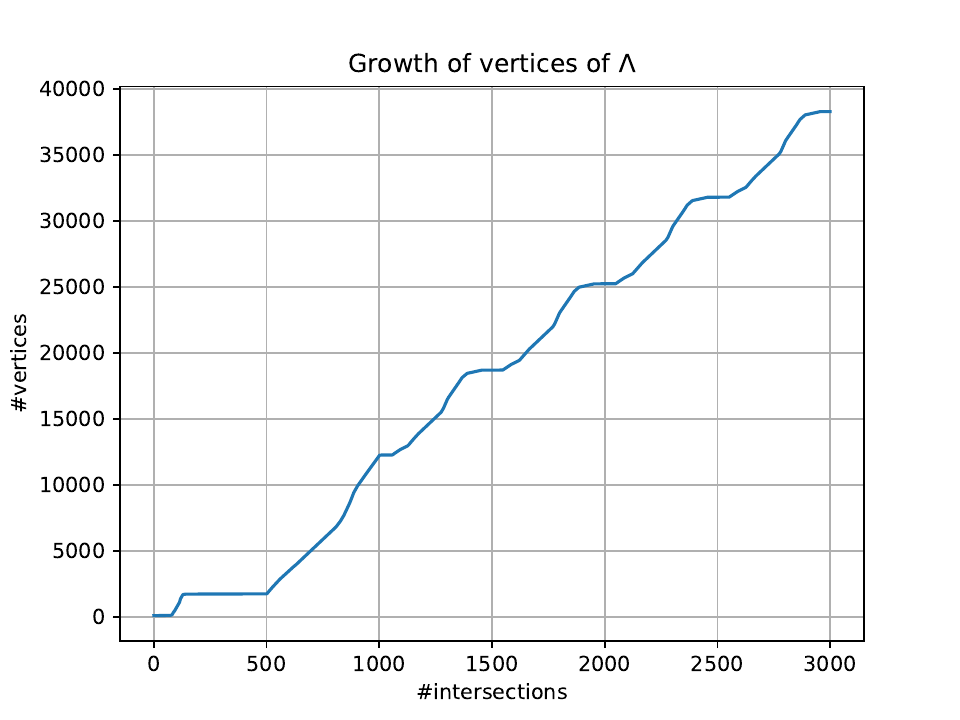}
	\caption{The number of vertices of $\Lambda$ in Algorithm \ref{area_sweep_alg} as a function of the number of intersections.}
	\label{number_vertices_growth}
\end{subfigure}
\hfill
\begin{subfigure}[t]{0.49\textwidth}
	\centering
	\includegraphics[width=0.95\textwidth]{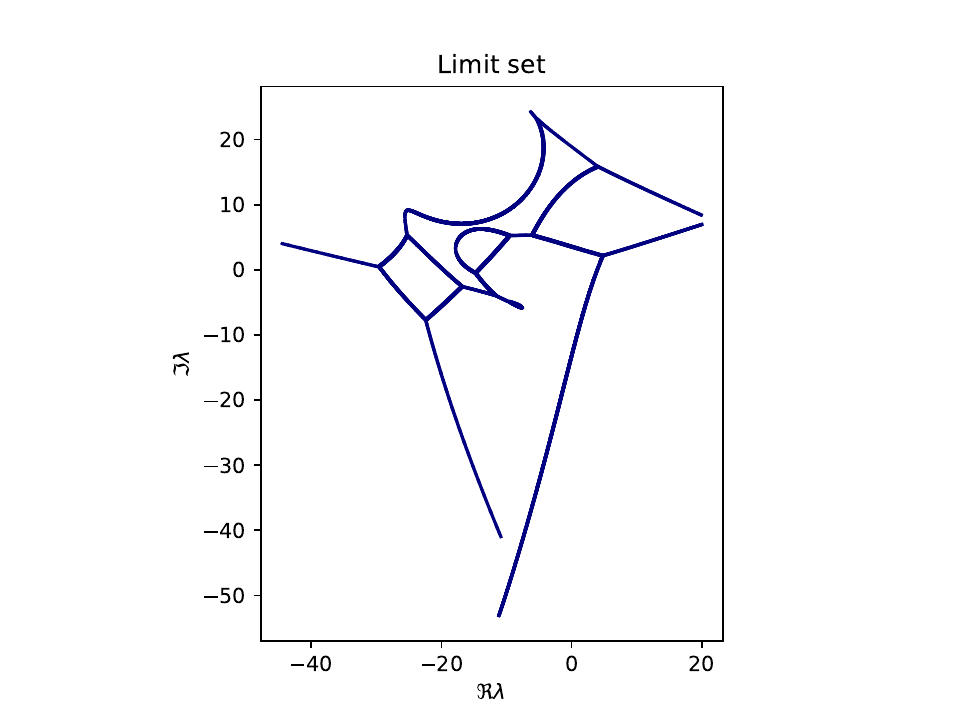}
	\caption{The approximating polygon for $\Lambda(b)$ given by Algorithm \ref{area_sweep_alg}.}
	\label{number_vertices_growth_limit_set}
\end{subfigure}
	\caption{Results of Algorithm \ref{area_sweep_alg} on  $b$ as in \eqref{random_pol}, $m = 500$, $n = 3000$,  $l = 500$, $\#sweeps=6$.}
\end{figure}

Assuming $w(i) = O(i + m)$, the total complexity is \begin{equation*}
	\begin{split}
		\MoveEqLeft \sum_{i=1}^{n} O\bigl(w(i) + m\bigr) + \sum_{s=1}^{\#sweeps} \sum_{i=1}^{l} O\bigr(w(s\cdot n /\#sweeps) + m\bigr) \\ & = O\left(\frac{n(1 + 2m + n + 2m)}{2}\right) + O\left(\frac{l \#sweeps(n/\#sweeps + 2m +n + 2m)}{2}\right) \\ & =  O(n^2 + nm + n l \#sweeps + m l \#sweeps ).
	\end{split}
\end{equation*} However, as we remarked earlier, we can practically optimize away the $\#sweeps$-factor if it is of moderate size, and in regular use cases it should never be set higher than about 10. Also, typically $l$ is chosen as smaller than $n$, else we would ``waste'' too many intersections on area sweeping. So with these reasonable restrictions on $\#sweeps$ and $l$, the time complexity for computing the approximating polygon in Algorithm $\ref{area_sweep_alg}$ should on average be $O(n^2 + nm)$. 

We now analyze the time complexity of Algorithm \ref{area_sweep_alg} when we also produce the approximating superset. Remember the earlier assumptions, that the number of vertices for ${P}_\rho$ is $O(m)$ and producing ${P}_\rho$ takes $O(m \log m)$ time. Using a similar argument as above, we could argue that the number of vertices for $\Lambda_b$ in Algorithm \ref{area_sweep_alg} grows linearly. So the time complexity for computing the approximating superset is almost the same as computing the approximating polygon, the only operation that adds additional time complexity is computing ${P}_\rho$. We compute $n$ ${P}_\rho$'s in Algorithm \ref{area_sweep_alg} so this step adds an additional $O(nm\log m)$ term to the overall time complexity. Hence we estimate the total time complexity for computing the approximating superset using Algorithm \ref{area_sweep_alg} to be $O(n^2 + mn) + O(nm\log m) = O(n^2 + mn \log m)$. Since computing the approximating polygon took $O(n^2 + mn)$ time, the total time complexity for running Algorithm 2 is $O(n^2 + mn \log m)$. 

Next we discuss the differences between the algebraic and geometric approach when the degree of the symbol varies. Simply put, the basic operation in the algebraic approach is to solve $b(z) = \lambda_k$, compared to the geometric approach where the basic operation is to evaluate $b(z)$. Since evaluation is inherently much simpler than root finding, we expect the geometric approach to perform better for high degree symbols. 

Above we have implicitly assumed that the cost of evaluating our symbol is constant, which is reasonable if we consider one symbol. However, one could quite easily account for the time complexity of evaluating the symbols which is directly linked to the degree $r+s$ of $b$. Let $p := r+s$ denote the degree of the symbol $b$ in Algorithm \ref{area_sweep_alg}, and analyze the evaluations of $b$. We notice that the evaluations occur when we evaluate $b_\rho^D(\text{vs})$ and $\norm{{x}_\rho''(v)}_\infty + \norm{{y}_\rho''(v)}_\infty$. One evaluation of $b$ takes $O(p)$ time if implemented correctly. The bulk of the evaluations will occur inside the \textbf{for}-loop on line 15, which will run $n$ times. And $\text{vs}$ has length $m$, so the total added complexity when accounting for $p$ is $O(pnm)$ and in total we get $O(n^2 + mn \log m + pnm)$. We can compare this with the algebraic approach, where we denote the number of sampled $\varphi$'s by $N$. When solving \eqref{graph_intersect_property} we will get $p$ candidates $\lambda_k$, and for each of those candidates we need to solve $b(z) = \lambda_k$. So the time complexity of the algebraic algorithm is $O(Nps(p))$, where $s(p)$ denotes the running time of the polynomial rootfinder used. Depending on the algorithm used, $s(p)$ will differ. However, we need to find $p$ roots, so $O(s(p)) \geq O(p)$. Hence the algebraic approach scales considerably faster in $p$ than the geometric approach. 

To compare the algebraic and geometric approach in practice, and show where the geometric algorithm has a slight advantage, we constructed the following example.

\begin{example}
	Consider \begin{equation}
		b(t) = t^{-400} + t^{-27} + t^{-26} + t^{-4} + t^{-3} + t.
		\label{high_order_symbol}
	\end{equation}
	We wish to compare the geometric and algebraic algorithm. And to make the comparison fair, we use a state of art polynomial solver, \code{MPSolve}, created by Bini, Riorentino and Robol, presented in \cite{mpsolve1} and \cite{mpsolve2}. We use the python bindings from the package \code{python-mpsolve} and use \code{solve} with the setting \code{STANDARD\_MPSOLVE}. We sample $N=10$ $\varphi$'s for the algebraic approach. We also run Algorithm \ref{area_sweep_alg} with $n=1000$, $m = 1000$, $l = 250$, and $\#sweeps=2$. The results can be seen below in Figure \ref{high_order_comparison}. Both algorithms took roughly 12 minutes to run. Comparing Figures \ref{high_order_subset} and \ref{high_order_superset}, one could argue that we get more information about $\Lambda(b)$ from the geometric algorithm, since \ref{high_order_superset} looks much more ``resolved'' than \ref{high_order_subset}. Also, one would expect an even larger discrepancy when considering even higher degree symbols. 
	
	\begin{figure}[H]
		\centering
		\begin{subfigure}[t]{0.49\textwidth}
			\centering
			\includegraphics[width=\textwidth]{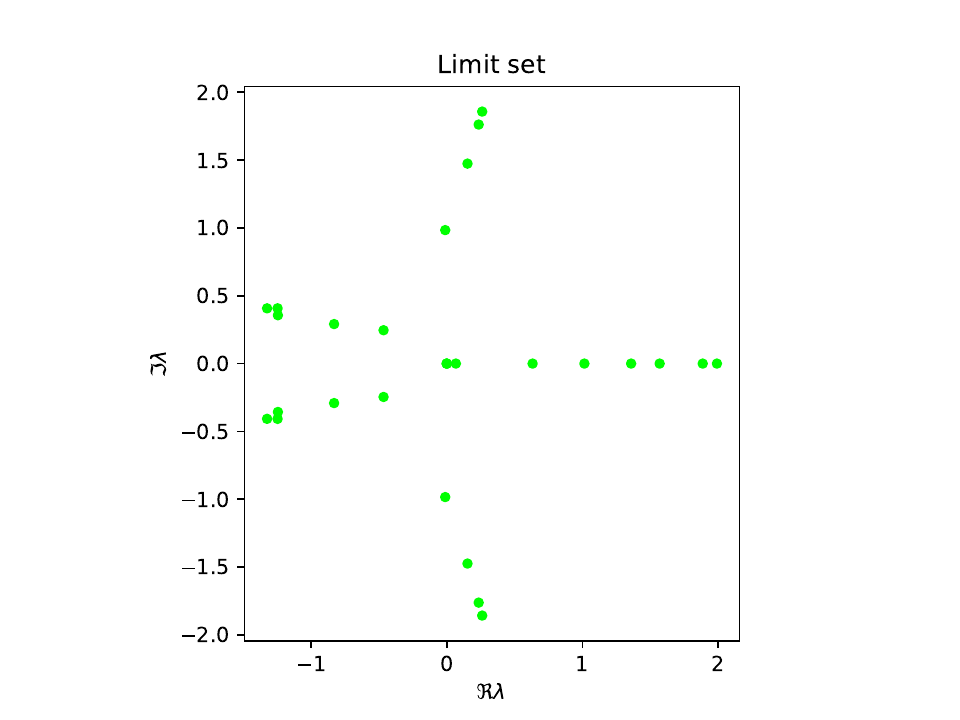}
			\caption{$\Lambda(b)$  estimated using the algebraic approach (from Section \ref{previous_work_subsection}) with $10$ $\varphi$'s sampled uniformly in $[0, \pi]$.}
			\label{high_order_subset}
		\end{subfigure}
		\hfill
		\begin{subfigure}[t]{0.49\textwidth}
			\centering
			\includegraphics[width=\textwidth]{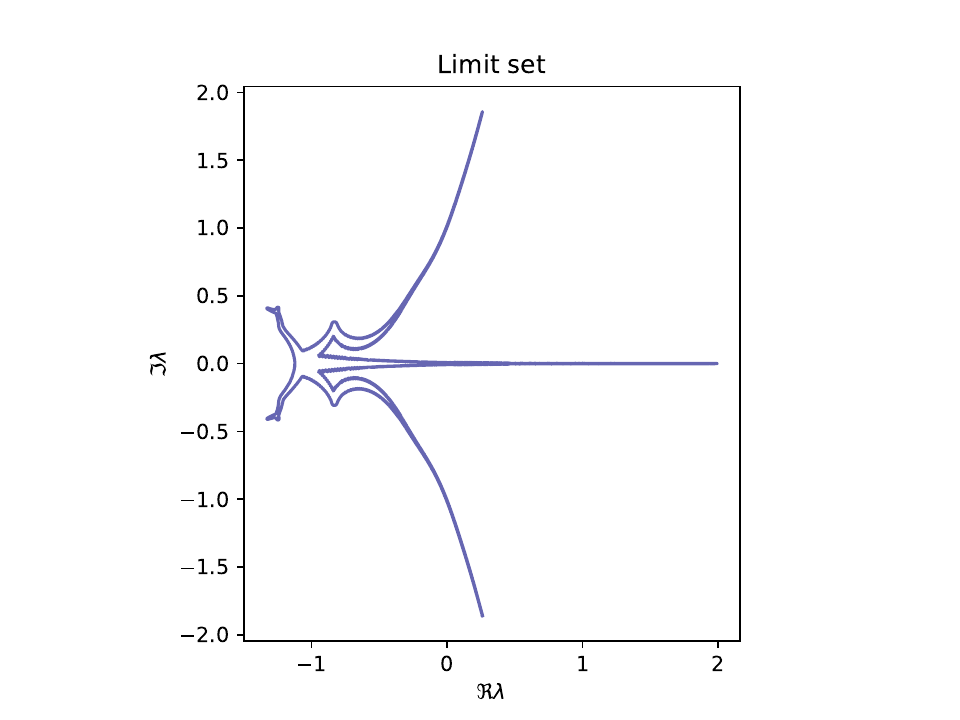}
			\caption{The approximating superset produced by Algorithm \ref{area_sweep_alg} when run on $b$ as in \eqref{high_order_symbol}, $n=1000$, $m = 1000$, $l = 250$, and $\#sweeps=2$.}
			\label{high_order_superset}
		\end{subfigure}
		
		\caption{A comparison between the algebraic and geometric approach. Both algorithms have been run on the symbol in \eqref{high_order_symbol} and took roughly 12 minutes to run.}
		\label{high_order_comparison}
	\end{figure}
	\label{high_order_example}
\end{example}

\subsection{Convergence speed}
To argue about the convergence speed we first remind ourselves about Remark \ref{different_errors_remark} which bounds the part of the error originating in the polygon approximation. Because of \eqref{error_separation}, we have a good idea of how big we must choose $m$ to get a certain error level. Therefore we will consider $m$ to be constant and reason about the convergence speed as $n$ increases. This approach also makes sense in terms of the algorithm, $m$ is set in the beginning and after that we can do as many intersections as we want. 

From \eqref{error_separation} we see that what we want to investigate is $$d(n) := d_H \Biggl(\,\bigcap_{j=0}^n \spec T(b_{\rho_j}), \Lambda(b)\Biggr).$$ So $d(n)$ describes how the error changes when $n$ increases. Consider Figure \ref{nbr_vert_illustration} but think of the black lines describing $\cap_{j=0}^n \spec T(b_{\rho_j})$, which makes sense since if $b'_\rho(t) \neq 0$, then locally $b_\rho(t)$ is almost linear. Note that the distances between the red curve and the red dots in Figure \ref{nbr_vert_illustration} is a lower bound of $d(n)$. If we increase $n$ by one, and intersect with the spectrum corresponding to the blue lines, we argue that the now relevant distances, the ones between the green dots and the red curve, should decrease geometrically. Assuming this is true for all of $\Lambda(b)$, we see that as $n$ is doubled, the maximum distance $d(n)$ should have decreased geometrically, since doubling $n$ corresponds to intersecting with a ``blue spectrum'' between each pair of black self-intersections, as in Figure \ref{nbr_vert_illustration}, and each black self-intersection corresponds to at most one $\rho$. So we have argued that \begin{equation}
	d(2n) \leq d(n)/r,
	\label{geometric_bound}
\end{equation} for some factor $r > 1$. If we assume the new intersections to be placed in the ``middle of the old ones'', we should have $r \approx 2$. 

However, ultimately we are interested in how $d(n)$ decreases as a function of the number of elementary operations. Let $k$ denote the number of elementary operations. As we argued in the previous subsection, $k = O(n^2 + mn \log m)$, and we assumed $m$ to be constant, so $k < (c^*)^2 n^2$, which means $n > \sqrt{k}/c^*$. Let $d^*(k)$ denote $d_H (\cap_{j=0}^n \spec T(b_{\rho_j}), \Lambda(b))$ as a function of the number of elementary operations. Further, we see that $d(n)$ and $d^*(k)$ are non-increasing, and get \begin{equation}
	d^*(k) \leq d\left(\floor*{\sqrt{k}/c^*}\right) \leq d\left(n_0 2^{\floor*{\log_2\left(\floor*{\frac{\sqrt{k}}{c^*}}/n_0\right)}} \right),
	\label{num_operations_ineq}
\end{equation} for some $n_0$, which is the $n$ at which we assume we begin our doubling argument. Using the bounds in \eqref{num_operations_ineq} and \eqref{geometric_bound} recursively, we obtain \begin{equation*}
\begin{multlined}
d^*(k) \leq \frac{d(n_0)}{r^{\floor*{\log_2\left(\floor*{\frac{\sqrt{k}}{c^*}}/n_0\right)}}} \leq \frac{d(n_0)}{2^{(\log_2 r) \left(\log_2\left(\floor*{\frac{\sqrt{k}}{c^*}}/n_0\right) - 1\right)}} \\ = \frac{d(n_0)}{\left(\floor*{\frac{\sqrt{k}}{c^*}}/2n_0\right)^{\log_2 r}} \leq \frac{c}{k^{\frac{1}{2}\log_2 r}}.
\end{multlined}
\end{equation*} for some constant $c$. The conclusion of the above argument is that for a fixed $m$ we should have $d^*(k) \leq O(1/k^{\frac{1}{2}\log_2 r})$. And if we -- as commented on earlier -- choose the $\rho$'s in a good way, $r$ should be $\approx 2$, which means $d^*(k) \leq O(1/\sqrt{k})$. 

So we have argued for $d_H (\cap_{j=0}^n \spec T(b_{\rho_j}), \Lambda(b) )$ decreasing as $O(1/\sqrt{k})$, and assuming that $\Delta_v$ is small enough, this argument also carries over to the distances between the approximating polygon and $\Lambda(b)$, and between the approximating superset and $\Lambda(b)$. This can be seen from \eqref{error_separation} and the proof of Corollary \ref{best_bound_result}.

However, this entire argument is based on the assumption that Figure \ref{intersection_illustration} is a good approximation of the local behavior of intersected spectra. An implicit assumption in Figure \ref{intersection_illustration} is for example that $b_\rho(\mathbb{T})$ is not tangent to itself (which is true for all but a finite number of points on $\Lambda(b)$). If one examines the symbol in Example \ref{main_example} more closely, one sees that this is exactly what happens in the center of ``the cross'', since \begin{equation*}
b_1(e^{i0}) = b_1(e^{i \frac{\pi}{2}}) = 0, \; \; \frac{\di }{\di v} b_1(e^{iv}) \Bigr|_{v = 0} = \frac{\di }{\di v} b_1(e^{iv}) \Bigr|_{v = \frac{\pi}{2}} = -(1+i).  
\end{equation*}

\begin{example}
Consider the same symbol as we did in the introduction $b(t) = t^{-4} + t$. We then have \begin{equation*}
	\Lambda(b) = \bigl\{\lambda \in \mathbb{C} : \lambda = re^{\frac{2\pi i}{5}}, \, 0 \leq r \leq 5 \cdot 4^{-4/5}\bigr\}.
\end{equation*} We want to test the arguments in this subsection empirically. Therefore we compute $d_H(\Lambda(b), \Lambda^{sup})$ with the approximating superset $\Lambda^{sup}$ computed using Algorithm \ref{first_approach_alg} for different $n$. Since we know $\Lambda(b)$ exactly we can compute $d_H(\Lambda(b), \Lambda^{sup})$ in a similar manner as we implemented Algorithm \ref{error_control}, since we can compute $\left(\Lambda(b)\right)_r$ for different $r$ and check whether $\left(\Lambda(b)\right)_r \supset \Lambda^{sup}$ or not. We compute $\Lambda^{sup}$ for $m = 2000$ and $n = 100, 200, \ldots, 1000$.

The results can be seen in Figure \ref{convergence_rate_fig}. If we first consider Figure \ref{convergence_rate_n}, this is roughly what we would expect from the above argument where one conclusion is that we should have $d(2n) \approx d(n)/2$. If we assume that $d_H(\Lambda(b), \Lambda^{sup}) \approx Cn^\alpha$, then we get $\alpha \approx -1$, which $-0.92$ is fairly close to. In Figure \ref{convergence_rate_time}, we should – according to our argument – after a while see $d_H(\Lambda(b), \Lambda^{sup})$ decaying as $t^{-0.5}$. This is also what we start seeing: in the end the decay is about $C t^{-0.63}$. The convex type curve we see in Figure \ref{convergence_rate_time} is also what we would expect for an algorithm of running time $O(n^2 + nm \log m)$. For small $n$, the second term is much larger since $m$ is much larger than $n$, so the time is almost linear in $n$.  
\begin{figure}[H]
	\centering	
	\begin{subfigure}[t]{0.49\textwidth}
		\centering
		\includegraphics[width=\textwidth]{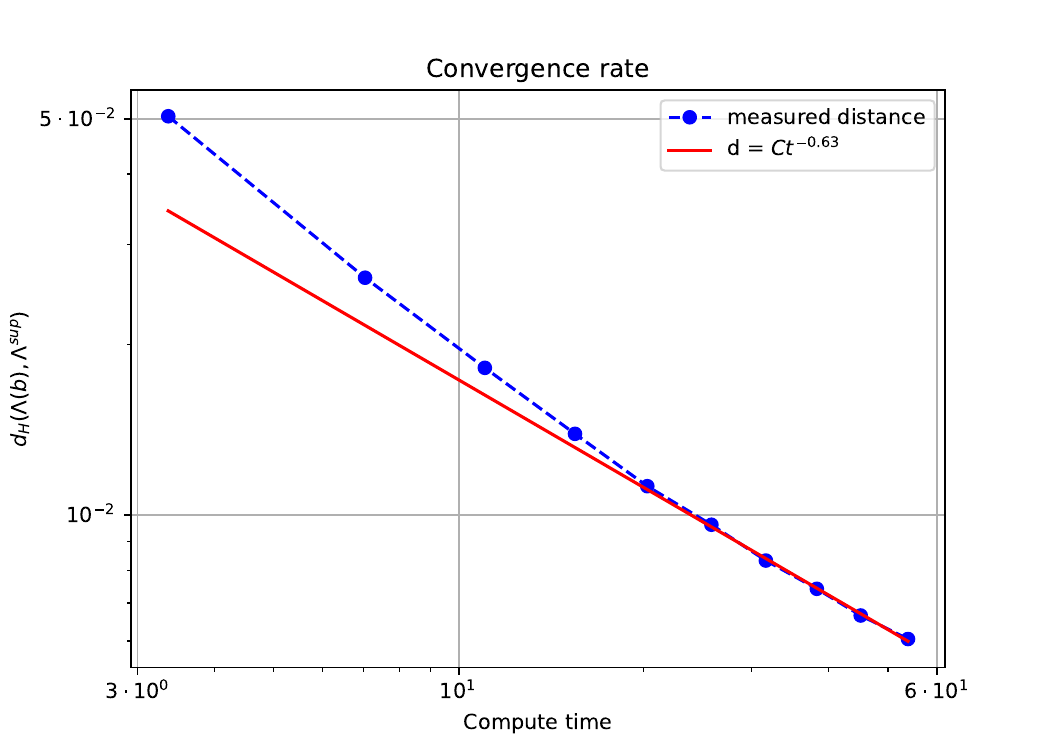}
		\caption{The Hausdorff distances $d_H(\Lambda(b), \Lambda^{sup})$ plotted as a function of the running time of the algorithm. The red line is fitted to the last five points.}
		\label{convergence_rate_time}
	\end{subfigure}
	\hfill
	\begin{subfigure}[t]{0.49\textwidth}
		\centering
		\includegraphics[width=\textwidth]{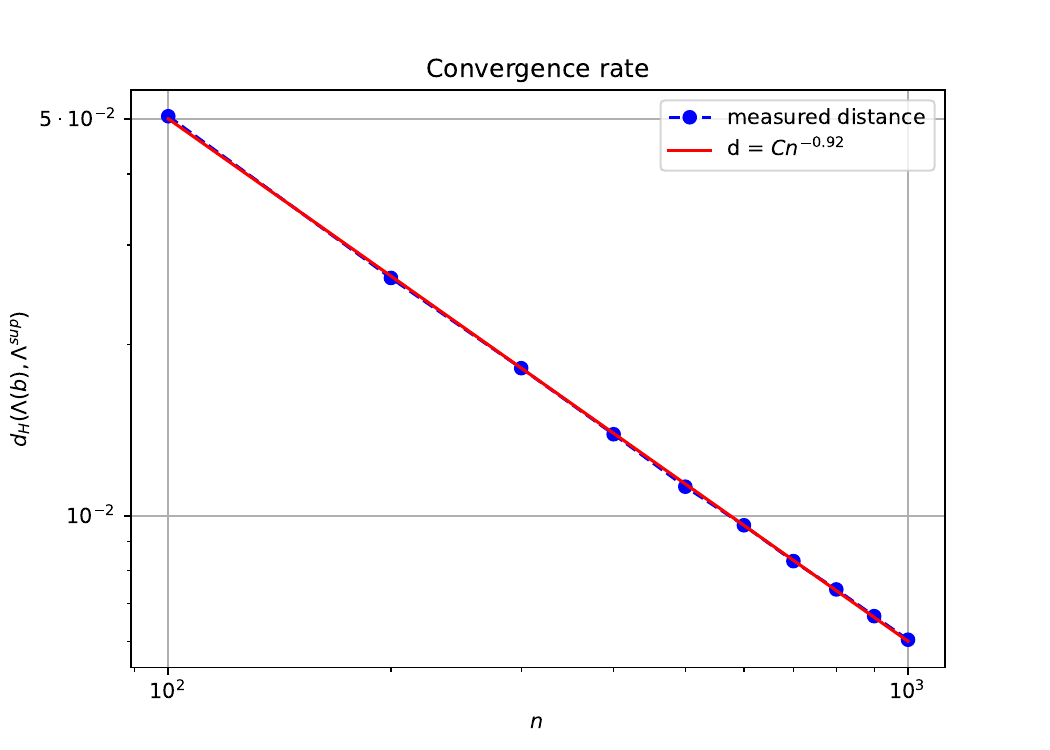}
		\caption{The Hausdorff distances $d_H(\Lambda(b), \Lambda^{sup})$ plotted as a function of $n$. The red line is fitted to all of the points.}
		\label{convergence_rate_n}
	\end{subfigure}

	\caption{Empirical distances $d_H(\Lambda(b), \Lambda^{sup})$ with the approximating superset $\Lambda^{sup}$ computed using Algorithm \ref{first_approach_alg} for $m = 2000$ and $n = 100, 200, \ldots, 1000$. Both subfigures have logarithmic axes.}
	\label{convergence_rate_fig}
\end{figure}

\end{example}

\section{Conclusion}
\label{conclusion_section}
A new approach to approximating $\Lambda(b)$ geometrically has been presented. Pseudocode for the approach can be found in Algorithm \ref{area_sweep_alg}, and the python code we used can be found in \cite{glisec}. Notably, we are able to produce an approximating superset for $\Lambda(b)$, and combining the superset with the subset from the previous approach we are able to compute an explicit error bound using Algorithm \ref{error_control}. 

Compared to the previous approaches that are algebraically based, the new approach does not require root finding of arbitrary complex polynomials, which as we saw in Example \ref{high_order_example} can be beneficial when dealing with symbols of high degree. So in a way, our new algorithm makes it feasible to compute with symbols of very large degree (1000 and upward), which was not feasible before.  

\section*{Acknowledgments}
The authors wish to thank Mikael Persson Sundqvist for fruitful and insightful discussions and Frank Wikström for suggesting the approach of constructing an approximating superset. The authors also want to thank the referees for many useful comments, in particular for suggesting the method used in Algorithm \ref{error_control} to explicitly bound the error.


\newpage

\bibliographystyle{plain}

\begin{thebibliography}{10}

\bibitem{glisec}
GLiSeC.
\newblock \url{https://github.com/TeodorBucht1729/GLiSeC}.
\newblock Accessed: 2024-01-10.

\bibitem{pyclipper}
Pyclipper.
\newblock \url{https://github.com/fonttools/pyclipper}.
\newblock Accessed: 2022-11-24.

\bibitem{beam_warming}
R. M. Beam and R. F. Warming.
\newblock {The asymptotic spectra of banded Toeplitz and quasi-Toeplitz
  matrices.}
\newblock {\em SIAM Journal on Scientific Computing} {\bf 14} (1993), 971--1006.

\bibitem{mpsolve1}
D. A. Bini and G. Fiorentino.
\newblock{Design, analysis, and implementation of a multiprecision polynomial rootfinder.}
\newblock{\textit{Numerical Algorithms} {\bf 23.2-3} (2000), 127--173.}

\bibitem{mpsolve2}
D. A. Bini and L. Robol.
\newblock{Solving secular and polynomial equations: A multiprecision algorithm.}
\newblock{\textit{Journal of Computational and Applied Mathematics} \textbf{272} (2014), 276--292.}

\bibitem{matrix_less_complex_eigenvals}
M. Bogoya,  S.-E. Ekström, S. Serra-Capizzano, and P. Vassalos.
\newblock{Matrix-less methods for the spectral approximation of large non-Hermitian Toeplitz matrices: A concise theoretical analysis and a numerical study.}
\newblock{\textit{Numerical Linear Algebra with Applications} (in press).}

\bibitem{pseudospectra_and_singular_values}
A. B{\"o}ttcher. \newblock{Pseudospectra and Singular Values of Large Convolution Operators.}
\newblock{\em {Journal of Integral Equations and Applications}}  {\bf 6(3)} (Summer 1994) 267--301. 

\bibitem{recent_structure}
A. B{\"o}ttcher, J. Gasca, S. M. Grudsky, and A. V. Kozak.
\newblock {Eigenvalue Clusters of Large Tetradiagonal Toeplitz Matrices.}
\newblock {\em Integral Equations and Operator Theory} {\bf 93} (2021) 1--27.

\bibitem{spectral_prop_bottcher}
A. B{\"o}ttcher and S. M. Grudsky.
\newblock {\em {Spectral properties of banded Toeplitz matrices.}}
\newblock Society for Industrial and Applied Mathematics, 2005.

\bibitem{spectrum_not_cont}
A. B{\"o}ttcher, S. M. Grudsky, and I. M. Spitkovsky.
\newblock {The spectrum is discontinuous on the manifold of Toeplitz
  operators.}
\newblock {\em Archiv der Mathematik} {\bf 75} (2000), 46--52.

\bibitem{bottcher1999introduction}
A. B{\"o}ttcher and B. Silbermann.
\newblock {\em {Introduction to large truncated Toeplitz matrices}}.
\newblock Springer, 1999.

\bibitem{hausdorff_reference}
E. {\v C}ech and M. Kat{\v e}tov.
\newblock {\em Point Sets}.
\newblock Academia, Publishing House of the Czechoslovak Academy of Sciences,
  1969.

\bibitem{offsetting_article}
X. Chen and S. McMains.
\newblock Polygon offsetting by computing winding numbers.
\newblock In {\em International Design Engineering Technical Conferences and
  Computers and Information in Engineering Conference}, volume 4739, pp
  565--575, 2005.

\bibitem{equilibrium_problem_banded}
M. Duits and A. B. J. Kuijlaars.
\newblock {An equilibrium problem for the limiting eigenvalue distribution of
  banded Toeplitz matrices}.
\newblock {\em SIAM Journal on Matrix Analysis and Applications} {\bf 30} (2008), 173--196.

\bibitem{matrix_less_real_eigs}
S.-E. Ekström and P. Vassalos.
\newblock{A matrix-less method to approximate the spectrum and the spectral function of Toeplitz matrices with real eigenvalues.}
\newblock{\textit{Numerical Algorithms} \textbf{89} (2022), 701–720.}


\bibitem{gohberg1952application}
I. Gohberg.
\newblock On an application of the theory of normed rings to singular integral
  equations.
\newblock {\em Uspekhi Matematicheskikh Nauk} {\bf 7} (1952), 149--156.

\bibitem{gohberg1958}
I. Gohberg.
\newblock On the number of solutions of homogeneous singular equations with
  continuous coefficients.
\newblock {\em Dokl. Akad. Nauk SSSR} {\bf 112} (1958), 327--330.

\bibitem{hirschman}
I. I. Hirschman.
\newblock {The spectra of certain Toeplitz matrices.}
\newblock {\em Illinois Journal of Mathematics} {\bf 11} (1967), 145--159.

\bibitem{hogben}
L. Hogben. 
\newblock{\em Handbook of Linear Algebra.} \newblock{Chapman and Hall/CRC, 2014.}



\bibitem{cont_spectrum}
I. S. Hwang and W. Y. Lee.
\newblock {On the continuity of spectra of Toeplitz operators.}
\newblock {\em Archiv der Mathematik} {\bf 70} (1998), 66--73.

\bibitem{clipper2}
A. Johnson.
\newblock {Clipper2 - Polygon Clipping and Offsetting Library}.
\newblock \url{http://www.angusj.com/clipper2/Docs/Overview.htm}.
\newblock Accessed: 2022-11-24.

\bibitem{perturbation_theory_kato}
T. Kat{\=o}.
\newblock {\em Perturbation theory for linear operators.}
\newblock Springer, 1995.

\bibitem{schmidt_spitzer}
P. Schmidt and F. Spitzer.
\newblock {The Toeplitz matrices of an arbitrary Laurent polynomial.}
\newblock {\em Mathematica Scandinavica} {\bf 8} (1960), 15--38.

\bibitem{toeplitz_original}
O. Toeplitz.
\newblock {Zur Theorie der quadratischen und bilinearen Formen von unendlichvielen veränderlichen.} 
\newblock {Mathematische Annalen} {\bf 70} (1911), 351--376.

\bibitem{embree_trefethen}
L. N. Trefethen and M. Embree.
\newblock{\em Spectra and Pseudospectra.}
\newblock{Princeton University Press, 2005.}

\bibitem{ullman}
J. L. Ullman.
\newblock {A problem of Schmidt and Spitzer.}
\newblock {\em Bulletin of the American Mathematical Society} {\bf 73} (1967), 883--885.

\bibitem{vatti_clipping}
B. R. Vatti.
\newblock A generic solution to polygon clipping.
\newblock {\em Communications of the ACM} {\bf 35} (1992), 56--63.

\bibitem{widom1990eigenvalue}
H. Widom.
\newblock {Eigenvalue distribution of nonselfadjoint Toeplitz matrices and the
  asymptotics of Toeplitz determinants in the case of nonvanishing index}.
\newblock {\em Operator Theory: Advances and Applications} {\bf 48} (1990), 387--421.

\bibitem{widom1994}
H. Widom.
\newblock {Eigenvalue Distribution for Nonselfadjoint Toeplitz Matrices}.
\newblock In E. L. Basor and I. Gohberg, editors, {\em Toeplitz Operators and
  Related Topics}, pp 1--8, Basel, 1994. Birkh{\"a}user Basel.

\end{thebibliography}

\end{document}